\documentclass[12pt]{amsart}

\usepackage{amsmath}
\usepackage{amssymb}
\usepackage[all]{xy}

\numberwithin{equation}{section}

\newtheorem{thm}{Theorem}[section]
\newtheorem{lem}[thm]{Lemma}
\newtheorem{prop}[thm]{Proposition}
\theoremstyle{definition}
\newtheorem*{ack}{Acknowledgements}
\newtheorem{cor}[thm]{Corollary}
\theoremstyle{definition}
\newtheorem{rem}[thm]{Remark}
\theoremstyle{definition}
\newtheorem{exam}[thm]{Example}
\newtheorem{exam-nota}[thm]{Example-Notation}

\theoremstyle{definition}
\newtheorem{dfn}[thm]{Definition}

\newtheorem{dfn-nota}[thm]{Definition-Notation}

\newtheorem{dfn-lem}[thm]{Lemma-Definition}
\newtheorem{dfn-thm}[thm]{Theorem-Definition}

\newcommand{\beqa}{\begin{eqnarray*}}
\newcommand{\eeqa}{\end{eqnarray*}}

\newcommand{\C}{\mbox{${\mathbb C}$}}

\newcommand{\Ad}{{\rm Ad}}

\newcommand{\fa}{\mbox{${\mathfrak a}$}}

\newcommand{\fk}{\mbox{${\mathfrak k}$}}

\newcommand{\fh}{\mbox{${\mathfrak h}$}}
\newcommand{\fn}{\mbox{${\mathfrak n}$}}

\newcommand{\fb}{\mbox{${\mathfrak b}$}}
\newcommand{\fz}{\mbox{${\mathfrak z}$}}
\newcommand{\fm}{\mbox{${\mathfrak m}$}}

\newcommand{\ad}{\text{ad}}

\newcommand{\Co}{\mathbb{C}}

\newcommand{\sfibregl}{\mathfrak{gl}}

\newcommand{\xifij}{\xi_{f_{i,j}}}

\newcommand{\Gamman}{\Gamma_{n}^{a_{1},a_{2}, \cdots , a_{n-1}}}
\newcommand{\Gammai}{\Gamma_{i}^{a_{1},a_{2}, \cdots , a_{i}}}

\newcommand{\Gprod}{Z_{1}\times\cdots\times Z_{n-1}}
\newcommand{\sol}{\Xi^{i}_{c_{i}, \, c_{i+1}}}

\newcommand{\nsol}{\Xi^{i}_{0,0}}
\newcommand{\glfibre}{\sfibregl(n)_{c}}
\newcommand{\glsfibre}{\sfibregl(n)^{sreg}_{c}}

\newcommand{\glnil}{\sfibregl(n)_{0}}
\newcommand{\glsnil}{\sfibregl(n)^{sreg}_{0}}
\newcommand{\fgl}{\mathfrak{gl}}

\newcommand{\Ct}{\Co^{\times}}
\newcommand{\orbi}{\mathcal{O}^{i}_{a_{i}}}
\newcommand{\orbu}{\mathcal{O}_{U}^{i}}
\newcommand{\orbl}{\mathcal{O}_{L}^{i}}

\newcommand{\nilrad}{\fn_{a_{1},\cdots, a_{n-1}}}
\newcommand{\fgltheta}{\fgl(n)_{\Theta}}
\newcommand{\fglomega}{\fgl(n)_{\Omega}}

\newcommand{\pci}{p_{c_{i}}(t)}
\newcommand{\pcp}{p_{c_{i+1}}(t)}
\newcommand{\orbj}{\mathcal{O}_{a_{j}}^{j}}
\newcommand{\torbi}{\mathcal{O}^{i}_{\widetilde{a_{i}}}}
\newcommand{\tGamman}{\Gamma_{n-1}^{\widetilde{a_{1}},\widetilde{a_{2}}, \cdots , \widetilde{a_{n-1}}}}

\begin{document}
\title[]{The Orbit Structure of the Gelfand-Zeitlin group on $n\times n$ matrices}


\author[M. Colarusso]{Mark Colarusso}
\email{mcolarus@nd.edu}
\subjclass[2000]{ Primary 14L30, 14R20, 37J35, 53D17} 
\address{Department of Mathematics, University of Notre Dame, Notre Dame, IN, 46556}
\maketitle
{\it{ Dedicated to Bertram Kostant on the occasion of his 80th birthday. }}

\begin{abstract}
In recent work (\cite{KW1},\cite{KW2}), Kostant and Wallach construct an action of a simply connected Lie group $A\simeq \Co^{{n\choose 2}}$ on $\fgl(n)$ using a completely integrable system derived from the Poisson analogue of the Gelfand-Zeitlin subalgebra of the enveloping algebra.  In \cite{KW1}, the authors show that $A$-orbits of dimension ${n\choose 2}$ form Lagrangian submanifolds of regular adjoint orbits in $\fgl(n)$.  They describe the orbit structure of $A$ on a certain Zariski open subset of regular semisimple elements.  In this paper, we describe all $A$-orbits of dimension ${n\choose 2}$ and thus all polarizations of regular adjoint orbits obtained using Gelfand-Zeitlin theory.
\end{abstract}

\section{Introduction}

In recent papers (\cite{KW1}, \cite{KW2}), Bertram Kostant and Nolan Wallach construct an action of a complex, commutative, simply connected Lie group $A\simeq \C^{{n\choose 2}}$ on the Lie algebra of $n\times n$ complex matrices $\fgl(n)$.  The dimension of this group is exactly half the dimension of a regular adjoint orbit in $\fgl(n)$ and orbits of $A$ of dimension ${n\choose 2}$ are Lagrangian submanifolds of regular adjoint orbits.  We refer to the group $A$ introduced by Kostant and Wallach as the Gelfand-Zeitlin group, because of its connection with the Gelfand-Zeitlin algebra, as we will explain in section \ref{s:GZ}.  


 The group $A$ and its action are constructed as follows.  Given $i<n$, we can think of $\fgl(i)\hookrightarrow \fgl(n)$ as a subalgebra by embedding an $i\times i$ matrix into the top left-hand corner of an $n\times n$ matrix. For $1\leq i\leq n$ and $1\leq j\leq i $, let $f_{i,j}(x)$ be the polynomial on $\fgl(n)$ defined by $f_{i,j}(x)=tr(x_{i}^{j})$, where $x_{i}$ denotes the $i\times i$ submatrix in the top left-hand corner of $x$.  In \cite{KW1}, it is shown that the functions $\{f_{i,j} |1\leq i\leq n,\, 1\leq j\leq i\}$ are algebraically independent and Poisson commute with respect to the Lie-Poisson structure on $\fgl(n)\simeq\fgl(n)^{*}$.  The corresponding Hamiltonian vector fields $\xifij$ generate a commutative Lie algebra $\fa$ of dimension ${n\choose 2}$.  The group $A$ is defined to be the simply connected, complex Lie group that corresponds to the Lie algebra $\fa$.  The vector fields $\xifij$ are complete (Theorem 3.5 in \cite{KW1}), and therefore $\fa$ integrates to a global action of $\C^{n\choose 2}$ on $\fgl(n)$.  This action of $\Co^{n\choose 2}$ defines the action of the group $A$ on $\fgl(n)$.  

Our goal in this paper is to describe all $A$-orbits of dimension ${n\choose 2}$.  An element $x\in\fgl(n)$ is called strongly regular if and only if its $A$-orbit is of dimension ${n\choose 2}$.  One way of studying such orbits is to study the action of $A$ on fibres the map $\Phi: \fgl(n)\to \C^{\frac{n(n+1)}{2}}$
\begin{equation}\label{eq:imoment}
\Phi(x)=(p_{1,1}(x_{1}), p_{2,1}(x_{2}), \cdots, p_{n,n}(x)),
\end{equation}
where $p_{i,j}(x_{i})$ is the coefficient of $t^{j-1}$ in the characteristic polynomial of $x_{i}$.

In Theorem 2.3 in \cite{KW1}, the authors show that this map is surjective and that every fibre of this map $\Phi^{-1}(c)=\fgl(n)_{c}$ contains strongly regular elements.  Following \cite{KW1}, we denote the strongly regular elements in the fibre $\glfibre$ by $\glsfibre$.  By Theorem 3.12 in \cite{KW1}, the $A$-orbits in $\fgl(n)^{sreg}$ are precisely the irreducible components of the fibres $\glsfibre$.  Thus, our study of the action of $A$ on $\fgl(n)^{sreg}$ is reduced to studying the $A$-orbit structure of the fibres $\glsfibre$.  In \cite{KW1}, Kostant and Wallach describe the $A$-orbit structure on a special class of fibres that consist of certain regular semisimple elements.  In this paper, we describe the $A$-orbit structure of $\glsfibre$ for any $c\in \C^{\frac{n(n+1)}{2}}$.  

In section \ref{s:GZ}, we describe the construction of the group $A$ in \cite{KW1} in more detail.  In section \ref{s:generics}, we describe the results in \cite{KW1} about its orbit structure.  We summarize theses results briefly here.  For any $x\in\fgl(i)$, let $\sigma(x)$ denote the spectrum of $x$.  In \cite{KW1}, Kostant and Wallach describe the action of the group $A$ on a Zariski open subset of regular semisimple elements defined by 
 $$
 \fglomega=\{x\in\fgl(n)|\;x_{i}\text { is regular semisimple}, \, \sigma(x_{i-1})\cap \sigma (x_{i})=\emptyset, \, 2\leq i\leq n\}.
$$
 
 Let $c_{i}\in\C^{i}$ and consider $c=(c_{1}, c_{2},\cdots, c_{n})\in\C^{1}\times\C^{2}\times\cdots\times\C^{n}=\C^{\frac{n(n+1)}{2}}$.  Regard $c_{i}=(z_{1},\cdots, z_{i})$ as the coefficients of the degree $i$ monic polynomial 
 \begin{equation}\label{eq:polyci}
 p_{c_{i}}(t)=z_{1}+z_{2} t+\cdots + z_{i} t^{i-1} +t^{i}.
 \end{equation}
 
 Let $\Omega_{n}$ denote the Zariski open subset of $\C^{\frac{n(n+1)}{2}}$ given by the tuples $c$ such that $p_{c_{i}}(t)$ has distinct roots and $p_{c_{i}}(t)$ and $p_{c_{i+1}}(t)$ have no roots in common.  Clearly, $\fglomega=\bigcup_{c\in\Omega_{n}} \glfibre$.  The action of $A$ on $\fglomega$ is described in the following theorem. (Theorem \ref{thm:generics}).

\begin{thm} \label{thm:igenerics}
The elements of $\fglomega$ are strongly regular.  If $c\in\Omega_{n}$ then $\glfibre=\glsfibre$ is precisely one orbit under the action of the group $A$.  Moreover, $\glfibre$ is a homogeneous space for a free, algebraic action of the torus $(\mathbb{C}^{\times})^{{n\choose 2}}$. 
\end{thm}

In section \ref{s:gamma}, we give a construction that describes an $A$-orbit in an arbitrary fibre $\glsfibre$ as the image of a certain morphism of a commutative, connected algebraic group into $\glsfibre$.  The construction in section \ref{s:gamma} gives a bijection between $A$-orbits in $\glsfibre$ and orbits of a product of connected, commutative algebraic groups acting freely on a fairly simple variety, but it does not enumerate the $A$-orbits in $\glsfibre$.   In section \ref{s:counting}, we use the construction developed in section \ref{s:gamma} and combinatorial data of the fibre $\glsfibre$ to give explicit descriptions of the $A$-orbits in $\glsfibre$.  The main result is Theorem \ref{thm:general}, which contrasts substantially with the generic case described in Theorem \ref{thm:igenerics}. 

\begin{thm}\label{thm:introgen}
 Let $c=(c_{1}, c_{2},\cdots, c_{n})\in\C^{1}\times\C^{2}\times\cdots\times\C^{n}=\C^{\frac{n(n+1)}{2}}$ be such that there are $0\leq j_{i}\leq i$ roots in common between the monic polynomials $p_{c_{i}}(t)$ and $p_{c_{i+1}}(t)$.  Then the number of $A$-orbits in $\glsfibre$ is exactly $2^{\sum_{i=1}^{n-1} j_{i}}.$  For $x\in\glsfibre$, let $Z_{i}$ denote the centralizer of the Jordan form of $x_{i}$ in $\fgl(i)$.  The orbits of $A$ on $\glsfibre$ are the orbits of a free algebraic action of the complex, commutative, connected algebraic group $Z=\Gprod$ on $\glsfibre$. 
\end{thm}

\begin{rem}\label{r:bp}
After the results of this paper were established, a very interesting paper by Roger Bielawski and Victor Pidstrygach appeared in \cite{BP} proving similar results.  The arguments are completely different, and the proofs were formed independently.  In \cite{BP}, the authors define an action of $A$ on the space of rational maps of fixed degree from the Riemann sphere into the flag manifold for $GL(n+1)$ and use symplectic reduction to obtain results about the strongly regular set.  They also show that there are $2^{\sum_{i=1}^{n-1} j_{i}}$ $A$-orbits in $\glsfibre$, $c$ as in Theorem \ref{thm:introgen}.  Our work differs from that of \cite{BP} in that we explicitly list the $A$-orbits in $\glsfibre$ and obtain an algebraic action of $\Gprod$ on $\glsfibre$ whose orbits are the same as those of $A$.  In spite of the relation between these papers, we feel that our paper provides a different and more precise perspective on the problem and deserves a place in the literature.

\end{rem}

The nilfibre $\glnil=\Phi^{-1}(0)$ contains some of the most interesting structure in regards to the action of $A$.  The fibre $\glnil$ has been studied extensively by Lie theorists and numerical linear algebraists.  Parlett and Strang \cite{PS} have studied matrices in $\glnil$ and have obtained interesting results.  Ovsienko \cite{Ov} has also studied $\glnil$, and has shown that it is a complete intersection.  It turns out that the $A$-orbits in $\glsnil$ correspond to $2^{n-1}$ Borel subalgebras of $\fgl(n)$.  The main results are contained in Theorems \ref{thm:nil} and \ref{thm:nilradical}.  We combine them into one single statement here.

\begin{thm}
The nilfibre $\glsnil$ contains $2^{n-1}$ $A$-orbits.  For $x\in\glsnil$, let $\overline{A\cdot x}$ denote the Zariski (=Hausdorff) closure of $A\cdot x$.  Then $\overline{A\cdot x}$ is a nilradical of a Borel subalgebra in $\fgl(n)$ that contains the standard Cartan subalgebra of diagonal matrices.
\end{thm}

The nilradicals obtained as closures of $A$-orbits in $\glsnil$ are described explicitly in Theorem \ref{thm:nilradical}.  We also describe the permutations that conjugate the strictly lower triangular matrices into each of these $2^{n-1}$ nilradicals in Theorem \ref{thm:perms}.  

Theorem \ref{thm:introgen} lets us identify exactly where the action of the group $A$ is transitive on $\glsfibre$.  (See Corollary \ref{c:generic} and Remark \ref{r:willsee}).  

\begin{cor}
The action of $A$ is transitive on $\glsfibre$ if and only if $p_{c_{i}}(t)$ and $p_{c_{i+1}}(t)$ are relatively prime for each $i$, $1\leq i\leq n-1$.  Moreover, for such $c\in \C^{\frac{n(n+1)}{2}}$ we have $\glfibre=\glsfibre$.
\end{cor}
This corollary allows us to identify the maximal subset of $\fgl(n)$ on which the action of $A$ is transitive on the fibres of the map $\Phi$ in (\ref{eq:imoment}) over this set.  The set $\fglomega$ is a proper open subset of this maximal set.  This is discussed in detail in section \ref{s:theta}.

\begin{ack}
This paper is based on the author's Ph.D. thesis \cite{Col}, written under the supervision of Nolan Wallach at the University of California, San Diego.  I would like to thank him for his patience and support as a thesis advisor.  

	I am also very grateful to Bertram Kostant with whom I had many fruitful discussions.  Our discussions helped me to further understand various aspects of Lie theory and helped me to form some of the results in my thesis.  I would also like to thank Sam Evens who taught me how to write a mathematical paper amongst many other things.
\end{ack}

\section{The Group $A$}\label{s:GZ}

We briefly discuss the construction of an analytic action of a group $A\simeq \C^{{n\choose 2}}$ on $\fgl(n)$ that appears in \cite{KW1} (see also \cite{Col2}).  

We view $\fgl(n)^{*}$ as a Poisson manifold with the Lie-Poisson structure (see \cite{Va}, \cite{CG}).  Recall that the Lie Poisson structure is the unique Poisson structure on the symmetric algebra $S(\fgl(n))=\C[\fgl(n)^{*}]$ such that if $x, \, y\in S^{1}(\fgl(n))$, then their Poisson bracket $\{ x, y\} =[x, y]$ is their Lie bracket.  We use the trace form to transfer the Poisson structure from $\fgl(n)^{*}$ to $\fgl(n)$.  For $i\leq n$, we can view $\fgl(i)\hookrightarrow \fgl(n)$ as a subalgebra, simply by embedding an $i\times i$ matrix in the top left-hand corner of an $n\times n$ matrix.
\begin{equation}\label{eq:subalg}
Y\hookrightarrow \left[\begin{array}{cc}
Y & 0\\
0 & 0\end{array}\right].
\end{equation}
We also have a corresponding embedding of the adjoint groups $GL(i)\hookrightarrow GL(n)$
$$
g\hookrightarrow \left [\begin{array}{cc} g & 0\\
0 & Id_{n-i}\end{array}\right ].
$$
For the purposes of this paper, we always think of $\fgl(i)\hookrightarrow\fgl(n)$ and $GL(i)\hookrightarrow GL(n)$ via these embeddings, unless otherwise stated.

We can use the embedding (\ref{eq:subalg}) to realize $\fgl(i)$ as a summand of $\fgl(n)$.  Indeed, we have
\begin{equation}\label{eq:perp}
\fgl(n)=\fgl(i)\oplus\fgl(i)^{\perp},
\end{equation}
where $\fgl(i)^{\perp}$ denotes the orthogonal complement of $\fgl(i)$ in $\fgl(n)$ with respect to the trace form.   It is convenient for us to have a coordinate description of this decomposition.  We make the following definition.

\begin{dfn}\label{def:cutoff}
For $x\in\fgl(n)$, we let $x_{i}\in\fgl(i)$ be the top left-hand corner of $x$, i.e. $(x_{i})_{k,l}=x_{k,l}$ for $1\leq k, l\leq i$.  We refer to $x_{i}$ as the $i\times i$ cutoff of $x$.\end{dfn}
Given a $y\in\fgl(n)$ its decomposition in (\ref{eq:perp}) is written $y=y_{i}\oplus y_{i}^{\perp}$ where $y_{i}^{\perp}$ denotes the entries $y_{k,l}$ where $k, l$ are not both in the set $\{1, \cdots, i\}$.  Using the decomposition in (\ref{eq:perp}), we can think of the polynomials on $\fgl(i)$, $P(\fgl(i))$, as a Poisson subalgebra of the polynomials on $\fgl(n)$, $P(\fgl(n))$.  Explicitly, if $f\in P(\fgl(i))$, (\ref{eq:perp}) gives $f(x)=f(x_{i})$ for $x\in\fgl(n)$.  The Poisson structure on $P(\fgl(i))$ inherited from $P(\fgl(n))$ agrees with the Lie-Poisson structure on $P(\fgl(i))$ (see \cite{KW1}, pg. 330).  

Since $\fgl(n)$ is a Poisson manifold, we have the notion of a Hamiltonian vector field $\xi_{f}$ for any holomorphic function $f\in\mathcal{O}(\fgl(n))$.  If $g\in\mathcal{O}(\fgl(n))$, then $\xi_{f}(g)=\{f, g\}$.  The group $A$ is defined as the simply connected, complex Lie group that corresponds to a certain Lie algebra of Hamiltonian vector fields on $\fgl(n)$.  
To define this Lie algebra of vector fields, we consider the subalgebra of $P(\fgl(n))$ generated by the adjoint invariant polynomials for each of the subalgebras $\fgl(i)$, $1\leq i\leq n$.
\begin{equation}\label{eq:GZ}
J(\fgl(n))=P(\fgl(1))^{GL(1)}\otimes\cdots\otimes P(\fgl(n))^{GL(n)}.
\end{equation}
This algebra may be viewed as a classical analogue of the Gelfand-Zeitlin subalgebra of the universal enveloping algebra $U(\fgl(n))$ (see \cite{DFO}).  
As $P(\fgl(i))^{GL(i)}$ is in the Poisson centre of $P(\fgl(i))$, it is easy to see that $J(\fgl(n))$ is Poisson commutative.  (See Proposition 2.1 in \cite{KW1}.)  Let $f_{i,1}, \cdots, f_{i, i}$ generate the ring $P(\fgl(i))^{GL(i)}$.  Then $J(\fgl(n))$ is generated by $\{f_{i,1}, \cdots, f_{i, i}| 1\leq i\leq n\}$.  Note that the sum
$$
\sum_{i=1}^{n-1} i =\frac{n(n-1)}{2}={n\choose 2}
$$
is half the dimension of a regular adjoint orbit in $\fgl(n)$.  We will see shortly that the functions $\{f_{i,1}, \cdots, f_{i, i}| 1\leq i\leq n-1\}$ form a completely integrable system on a regular adjoint orbit.  

The surprising fact about this integrable system proven by Kostant and Wallach in \cite{KW1} is that the corresponding Hamiltonian vector fields $\{\xifij |\;1\leq j\leq i$, $1\leq i\leq n-1\}$ are complete (see Theorem 3.5 in \cite{KW1}).  Let $f_{i,j}=tr(x_{i}^{j})$ and let $\fa=\{\xifij |\;1\leq j\leq i$, $1\leq i\leq n-1\}$.  We define $A$ as the simply connected, complex Lie group corresponding to the Lie algebra $\fa$.  Since the vector fields $\xifij$ commute for all $i$ and $j$, the corresponding (global) flows define a global action of $\C^{{n\choose 2}}$ on $\fgl(n)$.  $A\simeq\C^{{n\choose 2}}$, and it acts on $\fgl(n)$ by composing these flows in any order.  The action of $A$ also preserves adjoint orbits.  (See \cite{KW1}, Theorems 3.3, 3.4.)  


The action of $A\simeq\C^{{n\choose 2}}$ may seem at first glance to be non-canonical as choices are involved in its definition.  However, one can show that the orbit structure of $\C^{{n\choose 2}}$ given by integrating the complete vector fields $\xifij$ is independent of the choice of generators $f_{i,j}$ for $P(\fgl(i))^{GL(i)}$. (See Theorem 3.5 in \cite{KW1}.)  Since we are interested in studying the geometry of these orbits, we lose no information by fixing a choice of generators.  
\begin{rem}\label{r:ortho}
Using the Gelfand-Zeitlin algebra for complex orthogonal Lie algebras $\mathfrak{so}(n)$, we can define an analogous group, $\Co^{d}$ where $d$ is half the dimension of a regular adjoint orbit in $\mathfrak{so}(n)$.  The construction of the group and the study of its orbit structure on certain regular semisimple elements of $\mathfrak{so}(n)$ is discussed in detail in \cite{Col2}.
\end{rem}

For our choice of generators, we can write down the Hamiltonian vector fields $\xifij$ in coordinates and their corresponding global flows.  To do this, we use the following notation.  Given $x,\, z\in\fgl(n)$, we denote the directional derivative in the direction of $z$ evaluated at $x$ by $\partial_{x}^{z}$.  Its action on function on a holomorphic function $f$ is
\begin{equation}\label{eq:deriv}
\partial_{x}^{z} f=\frac{d}{dt} |_{t=0} f(x+tz).
\end{equation}
By Theorem 2.12 in \cite{KW1}
\begin{equation}\label{eq:coords}
 (\xi_{f_{i,j}})_{x}=\partial_{x}^{[  -jx_{i}^{j-1},\,x]}.
\end{equation}
  We see that $\xi_{f_{i,j}}$ integrates to an action of $\mathbb{C}$ on $\fgl(n)$ given by
 \begin{equation}\label{eq:flows}
 \Ad (\exp(tjx_{i}^{j-1}))\cdot x
\end{equation}
for $t\in\mathbb{C}$, where $x_{i}^{0}=Id_{i}\in\fgl(i)$. 

\begin{rem}\label{r:Aorb}
The orbits of $A$ are the composition of the (commuting) flows in (\ref{eq:flows}) for $1\leq i\leq n-1$, $1\leq j\leq i$ in any order acting on $x\in\fgl (n)$.  It is easy to see using (\ref{eq:flows}) that the action of $A$ stabilizes adjoint orbits.    
\end{rem}

Equation (\ref{eq:coords}) gives us a convenient description of the tangent space to the action of $A$ on $\fgl(n)$.  We first need some notation.  If $x\in\fgl(n)$, let $Z_{x_{i}}$ be the associative subalgebra of $\fgl(i)$ generated by the elements $Id_{i}, x_{i}, x_{i}^{2},\cdots, x_{i}^{i-1}$.  We then let $Z_{x}=\sum_{i=1}^{n} Z_{x_{i}}$.  Let $x\in\fgl(n)$ and let $A\cdot x$ denote its $A$-orbit.  Then equation (\ref{eq:coords}) gives us 
$$
T_{x}(A\cdot x)=span \{ (\xifij)_{x} | 1\leq i\leq n-1,\, 1\leq j\leq i\}=span\{ \partial_{x}^{[z,x]} | z\in Z_{x}\}.
$$

Following the notation in \cite{KW1}, we denote 
\begin{equation}\label{eq:dist}
V_{x}:=span\{ \partial_{x}^{[z,x]} | z\in Z_{x}\}=T_{x}(A\cdot x)\subset T_{x}(\fgl(n)).
\end{equation}

Our work focuses on orbits of $A$ of maximal dimension ${n\choose 2}$, as such orbits form Lagrangian submanifolds of regular adjoint orbits.  (If such orbits exist, they are the leaves of a maximal dimension of the Gelfand-Zeitlin integrable system.)  Accordingly, we make the following theorem-definition. (See Theorem 2.7 and Remark 2.8 in \cite{KW1}).  

\begin{dfn-thm}\label{d:sreg}
$x\in\fgl(n)$ is called strongly regular if and only if the differentials $\{(df_{i,j})_{x} | 1\leq i\leq n, 1\leq i \leq i\}$ are linearly independent at $x$.  Equivalently, $x$ is strongly regular if the $A$-orbit of $x$, $A\cdot x$ has $\dim(A\cdot x)={n\choose 2}$.   We denote the set of strongly regular elements of $\fgl(n)$ by $\fgl(n)^{sreg}$. 
\end{dfn-thm}

The goal of the paper is to determine the $A$-orbit structure of $\fgl(n)^{sreg}$.  In \cite{KW1}, Kostant and Wallach produce strongly regular elements using the map $\Phi: \fgl(n)\to \C^{\frac{n(n+1)}{2}}$,


\begin{equation}\label{eq:map}
\Phi(x)=(p_{1,1}(x_{1}), p_{2,1}(x_{2}),\cdots, p_{n,n}(x)), 
\end{equation}
where $p_{i,j}(x_{i})$ is the coefficient of $t^{j-1}$ in the characteristic polynomial of $x_{i}$.

One of the major results in \cite{KW1} is the following theorem concerning $\Phi$. (See Theorem 2.3 in \cite{KW1}.)
\begin{thm}\label{thm:Hess}
Let $\fb\subset\fgl(n)$ denote the standard Borel subalgebra of upper triangular matrices in $\fgl(n)$.  Let $f$ be the sum of the negative simple root vectors.  Then the restriction of $\Phi$ to the affine variety $f+\fb$ is an algebraic isomorphism.  
\end{thm}
We will refer to the elements of $f+\fb$ as Hessenberg matrices.  
They are matrices of the form
$$
f+\fb=\left [\begin{array}{ccccc}
a_{11} & a_{12} &\cdots & a_{1n-1} & a_{1n}\\
1 & a_{22} &\cdots & a_{2n-1} & a_{2n}\\
0 & 1 & \cdots & a_{3n-1}& a_{3n}\\
\vdots &\vdots &\ddots &\vdots &\vdots\\
0 & 0 &\cdots & 1 &a_{nn}\end{array}\right ]_{\mbox{\large .}}
$$

Note that Theorem \ref{thm:Hess} implies that if $x\in f+\fb$, then the differentials $\{(dp_{i,j})_{x} | 1\leq i\leq n, \,1\leq j\leq i\}$ are linearly independent.  
The sets of functions $ \{f_{i,j} | 1\leq i\leq n, \,1\leq j\leq i\}$ and $ \{p_{i,j} | 1\leq i\leq n, \,1\leq j\leq i\}$ both generate the classical analogue of the Gelfand-Zeitlin algebra $J(\fgl(n))$ (see (\ref{eq:GZ})).  It follows that for any $x\in\fgl(n)$, $span\{(df_{i,j})_{x} | 1\leq i\leq n, \,1\leq j\leq i\}=span\{(dp_{i,j})_{x} | 1\leq i\leq n, \,1\leq j\leq i\}$ by the Leibniz rule.  Theorem \ref{thm:Hess} then implies
$$
f+\fb \subset\fgl(n)^{sreg} \text{ and therefore } \fgl(n)^{sreg} \text{ is a non-empty Zariski open subset of } \fgl(n).
 $$  
 Thus, the functions 
 $$
 \{f_{i,j} | 1\leq i\leq n, \,1\leq j\leq i\}\text{  are algebraically independent.  }
 $$

For $c=(c_{1}, c_{2},\cdots, c_{n})\in\Co\times\Co^{2}\times\cdots\times \Co^{n}=\C^{\frac{n(n+1)}{2}}$ we denote the fibre $\Phi^{-1}(c)=\fgl(n)_{c}$, $\Phi$ as in (\ref{eq:map}).  For $c_{i}\in\Co^{i}$, we define a monic polynomial $p_{c_{i}}(t)$ with coefficients given by $c_{i}$ as in (\ref{eq:polyci}).   $x\in\glfibre$ if and only if $x_{i}$ has characteristic polynomial $p_{c_{i}}(t)$ for all $i$ .  Theorem \ref{thm:Hess} says that for any $c\in\C^{\frac{n(n+1)}{2}}$, $\fgl(n)_{c}$ is non-empty and contains a unique Hessenberg matrix.  
We denote the strongly regular elements of the fibre $\fgl(n)_{c}$, by $\glsfibre$ that is
$$
\glsfibre=\fgl(n)_{c}\cap\fgl(n)^{sreg}.
$$
Since Hessenberg matrices are strongly regular, we get 
$$
\glsfibre\text{ is a non-empty Zariski open subset of } \glfibre
$$
for any $c\in\C^{\frac{n(n+1)}{2}}$.  

Theorem \ref{thm:Hess} implies that every regular adjoint orbit contains strongly regular elements.  This follows from the fact that a regular adjoint orbit contains a companion matrix, which is Hessenberg.  We can then use $A$-orbits of dimension ${n\choose 2}$ to construct polarizations of dense, open submanifolds of regular adjoint orbits.  Hence, the Gelfand-Zeitlin system is completely integrable on each regular adjoint orbit (Theorem 3.36 in \cite{KW1}).

Our goal is to give a complete description of the $A$-orbit structure of $\fgl(n)^{sreg}$.  It follows from the Poisson commutativity of the algebra $J(\fgl(n))$ in (\ref{eq:GZ}) that the fibres $\glfibre$ are $A$-stable.  Whence, the fibres $\glsfibre$ are $A$-stable.   Moreover, Theorem 3.12 in \cite{KW1} implies that the $A$-orbits in $\fgl(n)^{sreg}$ are the irreducible components of the fibres $\glsfibre$.  From this it follows that
$$
\text{ there are only finitely many $A$-orbits in the fibre $\glsfibre$ .}
$$

In this paper, we describe the $A$-orbit structure of an arbitrary fibre $\glsfibre$ and count the exact number of $A$-orbits in the fibre.  This gives a complete description of the $A$-orbit structure of $\fgl(n)^{sreg}$.  

\begin{rem}
Note that the collection of fibres of the map $\Phi$ is the same as the collection of fibres of the moment map for the $A$-action $x\to (f_{1,1}(x_{1}), f_{2,1}(x_{2}),\cdots, f_{n,n}(x))$.  Thus, studying the action of $A$ on the fibres of $\Phi$ is essentially studying the action of $A$ on the fibres of the corresponding moment map.  We use the map $\Phi$ instead of the moment map, since it is easier to describe the fibres of $\Phi$.  
\end{rem}

For our purposes, it is convenient to have a more concrete characterization of strongly regular elements. (See Theorem 2.14 in \cite{KW1}.) 
\begin{prop}\label{prop:sreg}
Let $x\in\fgl(n)$ and let $\fz_{\fgl(i)}(x_{i})$ denote the centralizer in $\fgl(i)$ of $x_{i}$.  Then $x$ is strongly regular if and only if the following two conditions hold.\\
(a) $x_{i}\in\fgl(i)$ is regular for all $i$, $1\leq i\leq n$.  \\
(b) $\fz_{\fgl(i-1)}(x_{i-1})\cap \fz_{\fgl(i)}(x_{i})=0$ for all $2\leq i\leq n$.
\end{prop}


\section{The action of $A$ on generic matrices}\label{s:generics}
 For $x\in\fgl(i)$, let $\sigma(x)$ denote the spectrum of $x$, where $x$ is viewed as an element of $\fgl(i)$.  We consider the following Zariski open subset of regular semisimple elements of $\fgl(n)$
\begin{equation}
\fglomega=\{x\in\fgl(n)|\;x_{i}\text { is regular semisimple}, \, \sigma(x_{i-1})\cap \sigma (x_{i})=\emptyset, \, 2\leq i\leq n\}.
\end{equation}
Kostant and Wallach give a complete description of the action of $A$ on $\fglomega$.
We give an example of a matrix in $\fgl(3)_{\Omega}$.  
\begin{exam}\label{ex:omega}
Consider the matrix in $\fgl(3)$ 
$$
X=\left[\begin{array}{ccc}
1 & 2 & 16\\
1& 0 & 4\\
0 & 1 &-3\end{array}\right]_{\mbox{\large .}}
$$
We can compute that $X$ has eigenvalues $\sigma(X)=\{-3, 3, -2\}$ so that $X$ is regular semisimple and that $\sigma(X_{2})=\{2,-1\}$.  Clearly $\sigma(X_{1})=\{1\}$.  Thus $X\in \fgl(3)_{\Omega}$.
\end{exam}

We recall the notational convention introduced in (\ref{eq:polyci}). (If $c_{i}=(z_{1}, z_{2},\cdots, z_{i})\in\C^{i}$, then $p_{c_{i}}(t)=z_{1}+z_{2}t+\cdots + z_{i}t^{i-1} +t^{i}$.)  Let $\Omega_{n}\subset\C^{\frac{n(n+1)}{2}}$ be the Zariski open subset consisting of $c\in\C^{\frac{n(n+1)}{2}}$ with $c=(c_{1},\cdots, c_{i},\cdots, c_{n})$ such that $p_{c_{i}}(t)$ has distinct roots and $p_{c_{i}}(t)$ and $p_{c_{i+1}}(t)$ have no roots in common (remark 2.16 in \cite{KW1}).  It is easy to see that $\fglomega=\bigcup_{c\in\Omega_{n}} \glfibre$.  

Kostant and Wallach describe the $A$-orbit structure on  $\fglomega$ in Theorems 3.23 and 3.28 in \cite{KW1}.  We summarize the results of both of these theorems in one statement below. 
\begin{thm} \label{thm:generics}
The elements of $\fglomega$ are strongly regular.  If $c\in\Omega_{n}$, then $\glfibre=\glsfibre$ is precisely one orbit under the action of the group $A$.  Moreover, $\glfibre$ is a homogeneous space for a free, algebraic action of the torus $(\mathbb{C}^{\times})^{{n\choose 2}}$. 
\end{thm}

We sketch the ideas behind one possible proof of Theorem \ref{thm:generics} in the case of $\fgl(3)$.  For complete proofs and a more thorough explanation, see either \cite{KW1} or \cite{Col2}. 

For $x\in\fgl(3)$ its $A$-orbit is
\begin{equation}\label{eq:aact}
\Ad\left(\left[\begin{array}{ccc}
z_{1} & & \\
 & 1 & \\
 & & 1\end{array}\right] \left[\begin{array}{ccc}
z_{2} & & \\
 & z_{2} & \\ 
 & & 1\end{array}\right]\left[\begin{array}{ccc}
\multicolumn{2}{c}{\exp(tx_{2})} &  \\
 &  & \\  
 & & 1 \end{array}\right]\right)\cdot x,
 \end{equation}
 where $z_{1},\, z_{2}\in\Co^{\times}$ and $t\in\Co$. (See equation (\ref{eq:flows}).)  
 
 If we let $Z_{i}\subset GL(i)$ be the centralizer of $x_{i}$ in $GL(i)$, we notice from (\ref{eq:aact}) the action of $A$ appears to push down to an action of $Z_{1}\times Z_{2}$.  For $x\in\fgl(3)_{\Omega}$, we should then expect to see an action of $(\Co^{\times})^{3}$ as realizing the action of $A$.
 
Working directly from the definition of the action of $A$ in (\ref{eq:aact}) is cumbersome.  The action of $Z_{2}$ on $x_{2}$ for would be much easier to write down if $x_{2}$ were diagonal.  For $x\in\fgl(3)_{\Omega}$, $x_{2}$ is not diagonal, but it is diagonalizable.  So, we first diagonalize $x_{2}$ and then conjugate by the centralizer $Z_{2}=(\Ct)^{2}$.  If $\gamma(x)\in GL(2)$ is such $(\Ad(\gamma(x))\cdot x)_{2}$  is diagonal, then we can define an action of $(\Ct)^{3}$ on $\fgl(3)_{c}$ for $c\in\Omega_{3}$ by

%
\begin{equation}\label{eq:3act}
(z_{1}^{\prime}, z_{2}^{\prime}, z_{3}^{\prime}) \cdot x= \Ad\left(\left[\begin{array}{ccc}
z_{1}^{\prime} & & \\
 & 1 & \\
 & & 1\end{array}\right]\gamma(x)^{-1} \left[\begin{array}{ccc}
z_{2}^{\prime} & & \\
 & z_{3}^{\prime} & \\
 & & 1\end{array}\right] \gamma (x)\right)\cdot x,
\end{equation}
with $z_{i}^{\prime}\in\Ct$. 

We can show (\ref{eq:3act}) is a simply transitive algebraic group action on $\fgl(3)_{c}$ by explicit computation.  Comparing (\ref{eq:3act}) and (\ref{eq:aact}), it is not hard to believe that the action of $(\Ct)^{3}$ in (\ref{eq:3act}) has the same orbits as the action of $A$ on $\fgl(3)_{c}$.  To prove this precisely, one needs to see that $\fgl(3)_{c}^{sreg}=\fgl(3)_{c}$.  This can be proven by computing the tangent space to the action of $(\Co^{\times})^{3}$ in (\ref{eq:3act}) and showing that it is same as the subspace $V_{x}$ in (\ref{eq:dist}), or by appealing to Theorem 2.17 in \cite{KW1}.  The fact that $\fgl(3)_{c}$ is one $A$-orbit then follows easily by applying Theorem 3.12 in \cite{KW1}.  

This line of argument is not the one used in \cite{KW1} to prove Theorem \ref{thm:generics}.  The ideas here go back to a preliminary approach by Kostant and Wallach.  However, it is this method that generalizes to describe all orbits of $A$ in $\fgl(n)^{sreg}$.  We describe the general construction in the next section.

\section{Constructing non-generic $A$-orbits}\label{s:gamma}
\subsection{Overview}\label{s:gammasum}
In the next three sections, we classify $A$-orbits in $\fgl(n)^{sreg}$ by determining the $A$-orbit structure of an arbitrary fibre $\glsfibre$.  Let $c_{i}\in\C^{i}$ and $p_{c_{i}}(t)=(t-\lambda_{1})^{n_{1}}\cdots (t-\lambda_{r})^{n_{r}}$ with $\lambda_{j}\neq\lambda_{k}$ for $j\neq k$ (see \ref{eq:polyci}).  To study the action of $A$ on $\glfibre$ with $c=(c_{1}, \cdots, c_{i}, c_{i+1}, \cdots, c_{n})\in\Co^{1}\times\cdots\times\Co^{i}\times\Co^{i+1}\times\cdots\times\Co^{n}= \C^{\frac{n(n+1)}{2}}$, we consider elements of $\fgl(i+1)$ of the form
 

\begin{equation}\label{eq:bigmatrix}
	  \left[\begin{array}{cc}
\begin{array}{ccc}
\left[\begin{array}{cccc}
\lambda_{1}& 1 & \cdots & 0\\
0& \lambda_{1}&\ddots &\vdots\\
\vdots&\text{ }& \ddots & 1\\
0& \cdots &\cdots &\lambda_{1}\\
\end{array}\right ] & \multicolumn{2}{c}{0}\\
&\ddots& \\
\multicolumn{2}{c}{0}&\left[\begin{array}{cccc}
\lambda_{r}& 1 & \cdots & 0\\
0& \lambda_{r}&\ddots &\vdots\\
\vdots&\text{ }& \ddots & 1\\
0& \cdots &\cdots &\lambda_{r}\\
\end{array}\right ]
\end{array} & \begin{array}{c}
y_{1,1}\\
\vdots\\
\vdots\\
y_{1,n_{1}} \\
\vdots\\
y_{r,1}\\
\vdots\\
\vdots\\
y_{r,n_{r}}
\end{array}\\
\begin{array}{cccccccccc}
z_{1,1}&\cdots &\cdots& z_{1,n_{1}}&\cdots& &z_{r,1}&\cdots &\cdots& z_{r,n_{r}}
\end{array} & w
\end{array}\right]\\ 
\end{equation}
 with characteristic polynomial $p_{c_{i+1}}(t)$.  

To avoid ambiguity, it is necessary to order the Jordan blocks of the $i\times i$ cutoff of the matrix in (\ref{eq:bigmatrix}).   To do this, we introduce a lexicographical ordering on $\C$ defined as follows. Let $z_{1},\, z_{2}\in \C$, we say that $z_{1}>z_{2}$ if and only if $Re z_{1}> Re z_{2}$ or if $Re z_{1}=Re z_{2}$, then $Im z_{1}>Im z_{2}$.  

\begin{dfn}\label{def:sol}
Let $c_{i}\in\C^{i}$ be such that $p_{c_{i}}(t)=(t-\lambda_{1})^{n_{1}}\cdots (t-\lambda_{r})^{n_{r}}$ with $\lambda_{j}\neq \lambda_{k}$ (as in (\ref{eq:polyci})) and let $\lambda_{1}>\lambda_{2}>\cdots>\lambda_{r}$ in the lexicographical ordering on $\C$.  For $c_{i+1}\in \C^{i+1}$, we define $\sol$ as the set of elements $x\in\fgl(i+1)$ of the form (\ref{eq:bigmatrix}) whose characteristic polynomial is $p_{c_{i+1}}(t)$.   We refer to $\sol$ as the solution variety at level $i$.
\end{dfn}

 We know from Theorem \ref{thm:Hess} that $\sol$ is non-empty for any $c_{i}\in\C^{i}$ and any $c_{i+1}\in\C^{i+1}$.   Let us denote the regular Jordan form which is the $i\times i$ cutoff of the matrix in (\ref{eq:bigmatrix}) by $J$.  Let $Z_{i}$ denote the centralizer of $J$ in $GL(i)$.  As $J$ is regular, $Z_{i}$ is a connected, abelian algebraic group (see Proposition 14 in \cite{K}).  $Z_{i}$ acts algebraically on the solution variety $\Xi^{i}_{c_{i},c_{i+1}}$ by conjugation.  In the remainder of section \ref{s:gamma}, we give a bijection between $A$-orbits in $\glsfibre$ and free $\Gprod$ orbits on $\Xi^{1}_{c_{1}, c_{2}}\times\cdots\times \Xi^{n-1}_{c_{n-1}, c_{n}}$.  In Section \ref{s:counting}, we will classify the $Z_{i}$-orbits on $\sol$ using combinatorial data of the tuple $c\in \C^{\frac{n(n+1)}{2}}$.  We will then have a complete picture of the $A$ action on $\glsfibre$.  
 

 We now give a brief outline of the construction, which gives the bijection between $A$-orbits in $\glsfibre$ and $\Gprod$ orbits in $\Xi^{1}_{c_{1}, c_{2}}\times\cdots\times \Xi^{n-1}_{c_{n-1}, c_{n}}$.  This construction not only describes $A$-orbits in $\glsfibre$, but all $A$-orbits in the larger set $\glfibre\cap S$, where $S$ is the Zariski open subset of $\fgl(n)$ consisting of elements $x$ whose cutoffs $x_{i}$ for $1\leq i\leq n-1$ are regular.  We know by Proposition \ref{prop:sreg} (a) that $\glsfibre \subset \glfibre\cap S$, and it is in general a proper subset.  (See Example \ref{ex:nsreg} below.)


The construction proceeds as follows.  For $1\leq i\leq n-2$, we choose a $Z_{i}$-orbit $\orbi\in \sol$ consisting of regular elements of $\fgl(i+1)$.  For $i=n-1$, we choose any orbit $\mathcal{O}^{n-1}_{a_{n-1}}$ of $Z_{n-1}$ in $\Xi^{n-1}_{c_{n-1},c_{n}}$.   Then we define a morphism
$$
\Gamma_{n}^{a_{1},a_{2}, \cdots , a_{n-1}}: \mathcal{O}^{1}_{a_{1}}\times\cdots\times \mathcal{O}^{n-1}_{a_{n-1}}\rightarrow \fgl(n)_{c}\cap S.$$
 by
\begin{equation}\label{eq:gengamma}
\Gamma_{n}^{a_{1},a_{2}, \cdots , a_{n-1}}(x_{1},\cdots, x_{n-1})=
\Ad(g_{1,2}(x_{1})^{-1}g_{2,3}(x_{2})^{-1}
 \cdots g_{n-2,n-1}(x_{n-2})^{-1}) x_{n-1}.
\end{equation}
where $g_{i,i+1}(x_{i})$ conjugates $x_{i}$ into Jordan canonical form (with eigenvalues in descreasing lexicographical order).  We denote the image of the morphism $\Gamman$ by $Im\Gamman$.  

\begin{thm}\label{thm:gengamma}
Every $A$-orbit in $\fgl(n)_{c}\cap S$ is of the form $Im\Gamman$ for some choice of orbits $\orbi\subset\sol$ with $\orbi$ consisting of regular elements of $\fgl(i+1)$ for $1\leq i\leq n-2$.  
\end{thm}


In section \ref{s:gammageo}, we prove Theorem \ref{thm:gengamma} for $A$-orbits in $\glsfibre$ (see Theorem \ref{thm:Aorb}).  In section \ref{s:gengamma}, we establish the results needed to prove Thoerem \ref{thm:gengamma} for $\glfibre\cap S$.  



\subsection{Definition and properties of the $\Gamman$ maps}\label{s:gammamap}

We first define the map $\Gamman$ only for $Z_{i}$-orbits $\orbi\subset\sol$ on which $Z_{i}$ acts freely.  To define the map $\Gamman$, we must define a morphism $\orbi\to GL(i+1)$ which sends $y\to g_{i,i+1}(y)$, where $g_{i, i+1}(y)$ conjugates $y$ into Jordan form with eigenvalues in decreasing lexicographical order.  Since $Z_{i}$ acts freely on $\orbi$, we can identify $\orbi\simeq Z_{i}$ as algebraic varieties.  Let $x_{a_{i}}$ be an arbitrary choice of base point for the orbit $\orbi$, i.e. $\orbi=\Ad(Z_{i})\cdot x_{a_{i}}$.  We choose an element $g_{i,i+1}(x_{a_{i}})\in GL(i+1)$ that conjugates the base point $x_{a_{i}}$ into Jordan form (with eigenvalues in decreasing lexicographical order).  For $y=\Ad(k_{i})\cdot x_{a_{i}}$, with $k_{i}\in Z_{i}$, we define 
  \begin{equation}\label{eq:gmatrix}
  g_{i,i+1}(y)=g_{i,i+1}(x_{a_{i}}) k_{i}^{-1}.
  \end{equation}
    For each choice of orbit $\orbi\subset\sol$ for $1\leq i\leq n-1$, we define a morphism $\Gamman: Z_{1}\times\cdots\times Z_{n-1}\to \fgl(n)$,
\begin{equation}\label{eq:defn} 
\Gamma_{n}^{a_{1},\cdots, a_{n-1}}(k_{1},\cdots, k_{n-1})=
\Ad(k_{1}g_{1,2}(x_{a_{1}})^{-1}k_{2}g_{2,3}(x_{a_{2}})^{-1}
 \cdots k_{n-2} g_{n-2,n-1}(x_{a_{n-2}})^{-1} k_{n-1}) x_{a_{n-1}}.
 \end{equation}



We want to give a more intrinsic characterization of the image of the morphism $\Gamman$.  

\begin{prop}\label{r:sense}
The set $Im\Gamman\subset\glfibre\cap S$ and is equal to
\begin{equation}\label{eq:gammaset}
Im\Gamman=\{x\in\fgl(n) | \; x_{i+1}\in \Ad(GL(i))\cdot x_{a_{i}},\, \text{ for all } 1\leq i \leq n-1\}.
\end{equation}
Thus, $Im\Gamman$ is a quasi-affine subvariety of $\fgl(n)$.  
\end{prop}
The following simple observation is useful in proving Proposition \ref{r:sense}.  
\begin{rem}\label{r:stupid}
Let $x\in\glfibre\cap S$, and suppose that $g\in GL(i)$ is such that $\Ad(g)\cdot x=\Ad(g)\cdot x_{i}$ is in Jordan canonical form with eigenvalues in decreasing lexicographical order for $1\leq i\leq n-1$.  Then $[\Ad(g)\cdot x]_{i+1}=\Ad(g)\cdot x_{i+1}\in\sol$.  
\end{rem}

\begin{proof}[Proof of Proposition \ref{r:sense}]
For ease of notation, let us denote the set on the RHS of (\ref{eq:gammaset}) by $T$.  We note $T\subset\fgl(n)_{c}\cap S$.  Indeed, let $Y\in T$.  Then $Y_{i+1}\in \Ad(GL(i))\cdot x_{a_{i}}$ for $1\leq i\leq n-1$.  Since $x_{a_{i}}\in \Xi^{i}_{c_{i}, c_{i+1}}$, $Y_{i+1}$ has characteristic polynomial $p_{c_{i+1}}(t)$.  Also note that for $1\leq i\leq n-2$, $x_{a_{i}}$ is regular, and hence $Y_{i+1}$ is regular for $1\leq i\leq n-2$.   Lastly, using the fact that $k_{1}\in GL(1)=Z_{1}$ centralizes the $(1,1)$ entry of $x_{a_{1}}\in \Xi^{1}_{c_{1}, c_{2}}$, it follows that the $(1,1)$ entry of $Y$ is given by $c_{1}\in \Co$.

The inclusion $Im\Gamman\subset T$ is clear from the definition of $\Gamman$ in (\ref{eq:defn}).  To see the opposite inclusion we use induction.  Let $y\in T$.  Then $y_{2}\in\Ad(GL(1))\cdot x_{a_{1}}=\mathcal{O}^{1}_{a_{1}}$, since $Z_{1}=GL(1)$.  Thus, there exists a $k_{1}\in Z_{1}$ such that $y_{2}=\Ad(k_{1})\cdot x_{a_{1}}$.  It follows that $$z_{2}=[\Ad(g_{1,2}(x_{a_{1}}))\Ad(k_{1}^{-1})\cdot y]_{3}=[\Ad(g_{1,2}(x_{a_{1}}))\Ad(k_{1}^{-1})\cdot y_{3}]\in\Xi^{2}_{c_{2}, c_{3}}.$$  But $y_{3}\in \Ad(GL(2))\cdot x_{a_{2}}$, so that $z_{2}\in \Xi^{2}_{c_{2}, c_{3}}\cap \Ad(GL(2))\cdot x_{a_{2}}$.  From which it follows easily that $z_{2}\in  \mathcal{O}^{2}_{a_{2}}$.  Thus, there exists a $k_{2}\in Z_{2}$ such that $$[\Ad(g_{2,3}(x_{a_{2}}))\Ad(k_{2}^{-1})\Ad(g_{1,2}(x_{a_{1}}))\Ad(k_{1}^{-1})\cdot y]_{4}\in \Xi^{3}_{c_{3}, c_{4}}.$$  This completes the first two steps of the induction.  We now assume that there exist $k_{1},\cdots, k_{j-1}\in Z_{1}, \cdots, Z_{j-1}$, respectively such that 
 \begin{equation}\label{eq:shortcut}
z_{j}= [\Ad( g_{j-1,j}(x_{a_{j-1}}))\Ad(k_{j-1}^{-1})\cdots \Ad(g_{1,2}(x_{a_{1}}))  \Ad(k_{1}^{-1})\cdot y]_{j+1}\in\Xi^{j}_{c_{j}, c_{j+1}}.
 \end{equation}
  Since $y_{j+1}\in \Ad(GL(j))\cdot x_{a_{j}}$, it follows that $z_{j}\in\Xi^{j}_{c_{j}, c_{j+1}}\cap \Ad(GL(j))\cdot x_{a_{j}}$.   As above, it follows that $z_{j}\in \orbj$, so that there exists an element $k_{j}\in K_{j}$ such that $$[\Ad(g_{j,j+1}(x_{a_{j}}))\Ad(k_{j}^{-1})\Ad( g_{j-1,j}(x_{a_{j-1}}))\Ad(k_{j-1}^{-1})\cdots \Ad(g_{1,2}(x_{a_{1}}))  \Ad(k_{1}^{-1})\cdot y]_{j+2}\in \Xi^{j+1}_{c_{j+1}, c_{j+2}}.$$
We have made use of Remark \ref{r:stupid} throughout.  By induction, we conclude that there exist $k_{1},\cdots, k_{n-1}\in Z_{1},\cdots , Z_{n-1}$ respectively so that 
 $$
x_{a_{n-1}}= \Ad( k_{n-1}^{-1})\Ad( g_{n-2, n-1}(x_{a_{n-1}}))\Ad(k_{n-2}^{-1})\cdots \Ad(g_{1,2}(x_{a_{1}}))\Ad(k_{1}^{-1})\cdot y.
 $$ 
From which it follows that $y=\Gamma_{n}^{a_{1},\cdots, a_{n-1}}(k_{1},\cdots, k_{n-1})$.  

To see the final statement of the proposition, we observe $T$ is a Zariski locally closed subset of $\fgl(n)$.  Indeed, the set $U_{i}=\{x | x_{i+1}\in\Ad(GL(i))\cdot x_{a_{i}}\}$ is locally closed, since it is the preimage of the orbit $\Ad(GL(i))\cdot x_{a_{i}}\subset\fgl(i+1)$ under the projection morphism $\pi_{i+1}(x)=x_{i+1}$.  The set $T=U_{1}\cap \cdots\cap U_{n-1}$ is locally closed.  
\end{proof}

\begin{rem}
From Proposition \ref{r:sense} it follows that the set $Im\Gamman$ depends only on the orbits $\orbi$ for $1\leq i\leq n-1$, and is thus independent of the choices involved in defining the map $\Gamman$ in (\ref{eq:defn}). 
\end{rem}
\subsection{$\Gamman$ and $A$-orbits in $\glsfibre$}\label{s:gammageo}

In this section, we show that the image of the morphism $\Gamman$ is an $A$-orbit in $\glsfibre$.  The first step is to see $Im\Gamman$ is smooth variety.  

\begin{thm}\label{thm:embedd}
The morphism
$$
\Gamman : \Gprod\to \fgl(n)_{c}\cap S
$$ 
is an isomorphism onto its image.  Hence, $Im\Gamman$ is a smooth, irreducible subvariety of $\fgl(n)$ of dimension ${n\choose 2}$.  
\end{thm}
\begin{proof}

We explicitly construct an inverse $\Psi$ to $\Gamman$ and show that $\Psi: Im\Gamman\to \Gprod$ is a morphism.  Specifically, we show that there exist morphisms $\psi_{i}: Im\Gamman\to Z_{i}$ for $1\leq i\leq n-1$ so that the morphism 
\begin{equation}\label{eq:inverse}
\Psi=(\psi_{1},\cdots , \psi_{n-1}): Im\Gamman\to \Gprod
\end{equation}
is an inverse to $\Gamman$.  The morphisms $\psi_{i}$ are constructed inductively.  

Given $y\in Im\Gamman,\, y_{2}\in\mathcal{O}^{1}_{a_{1}}\subset\Xi^{1}_{c_{1},\, c_{2}}$ by Proposition \ref{r:sense}.   Thus, $y_{2}=\Ad(k_{1})\cdot x_{a_{1}}$ for a unique $k_{1}$ in $Z_{1}$.  The map $\mathcal{O}^{1}_{a_{1}}\to Z_{1}$ given by $\Ad(k_{1})\cdot x_{a_{1}}\to k_{1}$ is an isomorphism of smooth affine varieties.  Hence, the map $\psi_{1}(y)=k_{1}$ is a morphism.   

Arguing as in the proof of Proposition \ref{r:sense}, suppose that we have defined morphisms $\psi_{1}, \cdots , \psi_{j-1}$, with $\psi_{i}: Im\Gamman\to Z_{i}$ for $1\leq i\leq j-1$.  Then the function $Im\Gamman\to\mathcal{O}^{j}_{a_{j}}$ given by equation (\ref{eq:shortcut}),
$$
y\to [\Ad( g_{j-1,j}(x_{a_{j-1}}))\Ad(\psi_{j-1}(y)^{-1})\cdots \Ad(g_{1,2}(x_{a_{1}}))  \Ad(\psi_{1}(y)^{-1})\cdot y]_{j+1}
$$
is a morphism.  We can then define a morphism $\psi_{j}: Im\Gamman\to Z_{j}$ given by $\psi_{j}(y)=k_{j}$, where $k_{j}$ is the unique element of $Z_{j}$ such that 
\begin{equation}\label{eq:psij}
\Ad(k_{j})\cdot x_{a_{j}}=[\Ad( g_{j-1,j}(x_{a_{j-1}}))\Ad(\psi_{j-1}(y)^{-1})\cdots \Ad(g_{1,2}(x_{a_{1}}))  \Ad(\psi_{1}(y)^{-1})\cdot y]_{j+1}.
\end{equation}
%

This completes the induction.  

Now, we need to see that the map $\Psi$ is an inverse to $\Gamman$.   The fact that $\Gamman(\psi_{1}(y),\cdots ,\psi_{n-1}(y))=y$ follows exactly as in the proof of the inclusion $T\subset Im\Gamman$ in Proposition \ref{r:sense}.   


Finally, we show that $\Psi(\Gamman(k_{1},\cdots, k_{n-1}))=(k_{1},\cdots ,k_{n-1})$.  We make the following observation.  Consider the element
$$
\Ad(k_{j}g_{j,j+1}(x_{a_{j}})^{-1}
 \cdots g_{n-2,n-1}(x_{a_{n-2}})^{-1} k_{n-1}) \cdot x_{a_{n-1}}.
$$
The $(j+1)\times (j+1)$ cutoff of this element is equal to $k_{j}\cdot x_{a_{j}}$.  Using this fact with $j=1$, we have $\psi_{1}(y)=k_{1}$.  Assume that we have $\psi_{2}(y)=k_{1}, \cdots , \psi_{l}(y)=k_{l}$ for $2\leq l\leq j-1$.  Using the definition of $\psi_{j}$ in (\ref{eq:psij}), we obtain $$\Ad(\psi_{j}(y))\cdot x_{a_{j}}=[\Ad(k_{j})\Ad(g_{j,j+1}(x_{a_{j}})^{-1}\cdots g_{n-2,n-1}(x_{a_{n-2}})^{-1} k_{n-1}) x_{a_{n-1}}]_{j+1}=\Ad(k_{j})\cdot x_{a_{j}}.$$  Thus, by induction $\Psi\circ\Gamman=id$.  Hence, $\Psi$ is a regular inverse to the map $\Gamman$ and $\Psi$ is an isomorphism of varieties.  
\end{proof}

$Im\Gamman$ is a smooth, irreducible quasi-affine subvariety of $\fgl(n)$.  Thus, $Im\Gamman$ has the structure of a connected analytic submanifold of $\fgl(n)$, and $\Gamman$ is an analytic isomorphism.  We now show that the action of the analytic group $A$ preserves the submanifold $Im\Gamman$.



 
 
 \begin{prop}\label{prop:Aacts}
 The action of $A$ on $\fgl(n)$ preserves the submanifolds $Im\Gamman$.
 \end{prop}
 \begin{proof}
 We recall that the action of $A$ on $\fgl(n)$ is given by the composition of the flows in (\ref{eq:flows}) in any order. (See Remark \ref{r:Aorb}.)  Thus, to see that the action of $A$ preserves $Im\Gamman$ it suffices to see that the action of $\C$ in (\ref{eq:flows}) preserves $Im\Gamman$ for any $1\leq i\leq n-1$ and any $1\leq j\leq i$.  This can be seen easily using Proposition \ref{r:sense}.  Indeed, suppose that $x\in Im\Gamman$.   Then by Proposition \ref{r:sense}, $x_{k+1}\in \Ad(GL(k))\cdot x_{a_{k}}$ for any $1\leq k\leq n-1$.  Now we consider $\Ad(\exp (t jx_{i}^{j-1}))\cdot x$ as in (\ref{eq:flows}) with $t\in\C$ fixed.  For ease of notation let $h=\exp (t jx_{i}^{j-1})\in GL(i)$.  We claim $(\Ad(h)\cdot x)_{k+1}\in\Ad(GL(k))\cdot x_{a_{k}}$ for $1\leq k\leq n-1$.  We consider two cases.  Suppose $k\geq i$ and consider $(\Ad(h)\cdot x)_{k+1}$.  We have $(\Ad(h)\cdot x )_{k+1}=\Ad(h)\cdot x_{k+1}$.  But $x_{k+1}\in \Ad(GL(k))\cdot x_{a_{k}}$, so that $\Ad(h)\cdot x_{k+1}\in\Ad(GL(k))\cdot x_{a_{k}}$, as $GL(i)\subset GL(k)$.   
Next, we suppose that $k<i$, so that $k+1\leq i$.  Since $h\in GL(i)$ centralizes $x_{i}$, 
$$
(\Ad(h) x)_{k+1}=(\Ad(h) (x_{i}))_{k+1}=(x_{i})_{k+1}= x_{k+1}\in\Ad(GL(k))\cdot x_{a_{k}}.
$$
By Proposition \ref{r:sense} $\Ad(h)\cdot x\in Im\Gamman$.  This completes the proof.  
 \end{proof}
 
 Before stating the main theorem of this section, we need to state a technical result about the action of $Z_{i}$ on the solution varieties $\sol$.  This result will be proven independently of the following theorem in section \ref{s:gengamma}.  
 \begin{lem}\label{l:conn}
 For $x\in\sol$, the isotropy group of $x$ under the action of $Z_{i}$, $Stab(x)$, is a connected algebraic group.
 \end{lem}
 
 Thus, given an orbit of $Z_{i}$, $\mathcal{O}\subset \sol$ 
 \begin{equation}\label{eq:freely}
 \dim(\mathcal{O})=i\text{ if and only if } Z_{i} \text{ acts freely on } \mathcal{O}.
 \end{equation}
 
We are now ready to prove the main theorem of this section. 
\begin{thm}\label{thm:Aorb}
The submanifold $Im\Gamman\subset \glfibre\cap S$ is a single $A$-orbit in $\glsfibre$.  Moreover every $A$-orbit in $\glsfibre$ is of the form $Im\Gamman$ for some choice of free $Z_{i}$-orbits $\mathcal{O}^{i}_{a_{i}}\subset\sol$ with $\orbi\subset \fgl(i+1)^{reg}$, for $1\leq i\leq n-1$. 
\end{thm}
\begin{proof}
First, we show that $Im\Gamman$ is an $A$-orbit.  For this, we need to describe the tangent space $T_{y}(Im\Gamman)=(d\Gamman)_{\underline{k}}$, where $\underline{k}=(k_{1},\cdots, k_{n-1})\in\Gprod$ and $y=\Gamman(\underline{k})$.  Let $\{ \alpha_{i1},\cdots, \alpha_{ii}\}$ be a basis for $Lie(Z_{i})=\fz_{i}$.  Working analytically, we compute  
\begin{equation}\label{eq:firstdiff}
(d\Gamman)_{\underline{k}}(0,\cdots, \alpha_{ij},\cdots, 0)=\frac{d}{dt}|_{t=0}\; \Gamman(k_{1},\cdots, k_{i}\exp(t\alpha_{ij}), \cdots , k_{n-1}),
\end{equation}
for $1\leq j\leq i$.  Using the definition of the morphism $\Gamman$ the RHS of (\ref{eq:firstdiff}) becomes 
\begin{equation}\label{eq:secdiff}
\frac{d}{dt}|_{t=0} \Ad(k_{1}g_{1,2}(x_{a_{1}})^{-1}\cdots k_{i} \exp (t\alpha_{ij})g_{i,i+1}(x_{a_{i}})^{-1}
 \cdots k_{n-2} g_{n-2,n-1}(x_{a_{n-2}})^{-1} k_{n-1}) x_{a_{n-1}}.
\end{equation}
Let 
\begin{equation}\label{eq:li}
l_{i}=k_{1}g_{1,2}(x_{a_{1}})^{-1}\cdots k_{i}
 \text{ and let } h_{i}=g_{i,i+1}(x_{a_{i}})^{-1} \cdots k_{n-2} g_{n-2,n-1}(x_{a_{n-2}})^{-1} k_{n-1}.
 \end{equation}
 Then we can write (\ref{eq:secdiff}) as 
 $$
 \frac{d}{dt}|_{t=0}\Ad(l_{i}\exp(t \alpha_{ij}) h_{i})\cdot x_{a_{n-1}},
 $$
 which has differential
 \begin{equation}\label{eq:finaldiff}
 \ad(\Ad(l_{i})\cdot\alpha_{ij} )\cdot( \Ad(l_{i}h_{i})\cdot x_{a_{n-1}}).
 \end{equation}
 By definition of the element $l_{i}\in GL(i)$, the $i\times i$ cutoff of $\Ad(l_{i}^{-1}) \cdot y=\Ad(l_{i}^{-1}) \cdot y_{i}$ is in Jordan form (with eigenvalues in decreasing lexicographical order).  Hence elements of the form $\Ad(l_{i})\cdot \alpha_{ij}=\gamma_{ij}$ for $1\leq j\leq i$ form a basis for $\fz_{\fgl(i)} (y_{i})$.  Since $\Ad(l_{i}h_{i})\cdot x_{a_{n-1}}=y$, (\ref{eq:finaldiff}) implies the image of $d(\Gamman)_{\underline{k}}$ is
 \begin{equation}\label{eq:tanspace}
Im(d\Gamman)_{\underline{k}}=span\{\partial_{y}^{[\gamma_{i,j},y]},\; 1\leq i\leq n-1, 1\leq j\leq i\}=T_{y}(Im\Gamman).
 \end{equation}

We recall equation (\ref{eq:dist}), 
$$
T_{y}(A\cdot y)=span\{\partial_{y}^{ [z,y]} | z\in Z_{y}\}: =V_{y}.
$$
 Now, $y\in Im\Gamman$ has the property that $y_{i}$ is regular for all $i\leq n-1$, so that $\fz_{\fgl(i)}(y_{i})$ has basis $\{Id_{i}, y_{i}, \cdots , y_{i}^{i-1}\}$ (see \cite{K}, pg 382).   
 Thus,  
 \begin{equation}\label{eq:equalities}
  T_{y}(Im\Gamman)=span\{ \partial_{y}^{[z,y]} | z\in Z_{y}\}
=V_{y}.
\end{equation}
Equation (\ref{eq:equalities}) gives
 \begin{equation}\label{eq:dimensions}
 \dim V_{y}=\dim (A\cdot y)={n\choose 2},
 \end{equation}
 which implies $Im\Gamman\subset\glsfibre$.  By Proposition \ref{prop:Aacts}, $A$ acts on $Im\Gamman$.  We claim that the action of $A$ is transitive on $Im\Gamman$.  Indeed, given an $A$-orbit $A\cdot y$ with $y\in Im\Gamman$, $A\cdot y\subset Im\Gamman$ is a submanifold of the same dimension as $Im\Gamman$ by (\ref{eq:dimensions}), and thus must be open.  The action of $A$ is then clearly transitive on $Im\Gamman$, as $Im\Gamman$ is connected.  

We now show that every $A$-orbit in $\glsfibre$ is obtained in this manner.  For $x\in\glsfibre$, by part (a) of Proposition \ref{prop:sreg} and Remark \ref{r:stupid} there exists a matrix $g_{i}\in GL(i)$ such that $z_{i}=Ad(g_{i})\cdot x_{i+1}\in\sol$ and $z_{i}$ is regular for each $1\leq i\leq n-1$.  Thus $z_{i}\in\orbi$, with $\orbi$ an orbit of $Z_{i}$ in $\sol$ consisting of regular elements of $\fgl(i+1)$.  We claim that $Z_{i}$ must act freely on $\orbi$.  We suppose to the contrary that $Stab(x_{a_{i}})$ is non-trivial.  Lemma \ref{l:conn} gives that $\dim(Stab(x_{a_{i}}))\geq 1$.  But, this implies $\dim(Z_{GL(i)}(x_{i})\cap Z_{GL(i+1)}(x_{i+1}))\geq 1$, contradicting part (b) of Proposition \ref{prop:sreg}.   By Proposition \ref{r:sense} $x\in Im\Gamman$ for some choice of free $Z_{i}$-orbits $\orbi\subset \sol$.  This completes the proof of the theorem.


\end{proof}

\begin{rem}\label{r:different}
Let $\Gamman$ be defined using $Z_{i}$-orbits, $\orbi$ and let $\tGamman$ be defined using $Z_{i}$-orbits $\torbi=\Ad(Z_{i})\cdot x_{\widetilde{a_{i}}}$, where for some $i$, $1\leq i\leq n-1$, $\orbi\cap \torbi=\emptyset$.  Then it follows from Proposition \ref{r:sense} that the $A$-orbits $Im\Gamman$ and $Im\tGamman$ are distinct.  Indeed, suppose to the contrary that $y\in Im\Gamman\cap Im\tGamman$.  By Proposition \ref{r:sense}, we have $y_{i+1}\in \Ad(GL(i))\cdot x_{a_{i}}\cap \Ad(GL(i))\cdot x_{\widetilde{a_{i}}}$.  This implies that there exists $h\in GL(i)$ such that $\Ad(h)\cdot x_{a_{i}}=x_{\widetilde{a_{i}}}$. Since $x_{a_{i}}, \, x_{\widetilde{a_{i}}}\in \sol$, the previous equation forces $h\in Z_{i}$, which implies $\orbi=\torbi$, a contradiction.  We have thus established a bijection between free $\Gprod$ orbits on the product of solution varieties $\Xi^{1}_{c_{1},\, c_{2}}\times\cdots\cdots\times \Xi^{n-1}_{c_{n-1},\, c_{n}}$ and $A$-orbits in $\glsfibre$. 

\end{rem}

On the subvariety $Im\Gamma_{n}^{a_{1},\cdots, a_{n-1}}$, we have a free and transitive algebraic action of the algebraic group $Z=\Gprod$.  This action is defined by the following formula.   

\begin{equation}\label{eq:bigact} \text{If } (\Gamma_{n}^{a_{1},a_{2}, \cdots , a_{n-1}})^{-1}(y)=(k_{1},\cdots ,k_{n-1})\text{, then }
(k_{1}^{\prime},\cdots, k_{n-1}^{\prime})\cdot y=\Gamma_{n}^{a_{1},a_{2}, \cdots , a_{n-1}}(k_{1}^{\prime} k_{1},\cdots, k_{n-1}^{\prime} k_{n-1}).
\end{equation}

  \begin{rem}
 The action in (\ref{eq:bigact}) is the generalization of the action of $(\Ct)^{3}$ in (\ref{eq:3act}) to the non-generic case.  
  \end{rem}
  Thus, the $A$-orbit $Im\Gamman$ is the orbit of an algebraic group acting on a quasi-affine variety.  We now show that $Z=\Gprod$ acts algebraically on the fibre $\glsfibre$.  By Theorem 3.12 in \cite{KW1} the $A$-orbits in $\glsfibre$ are the irreducible components of $\glsfibre$.  Since they are disjoint, these components are both open and closed in $\glsfibre$ (in the Zariski topology on $\glsfibre$).  Following \cite{KW1}, we index these components by $\fgl^{sreg}_{c,i}(n)=A\cdot x(i)$, with $x(i)\in \glsfibre$.  Now, we have morphisms $\phi_{i}: Z\times\fgl^{sreg}_{c,i}(n)\rightarrow \fgl^{sreg}_{c}(n)$ given by the action of $Z$ on $Im\Gamma_{n}^{a_{1},\cdots , a_{n-1}}$.  The sets $Z\times \fgl^{sreg}_{c,i}(n)$ are (Zariski) open in the product $Z\times \glsfibre$ and are disjoint.  Thus, the morphisms $\phi_{i}$ glue to a unique morphism
	$$
	\Phi: Z\times \glsfibre\rightarrow \glsfibre
	\text{ such that }
	 \Phi | _{Z\times \fgl^{sreg}_{c,i}(n)}=\phi_{i}.
	 	 $$
	 The morphism $\Phi$ defines an algebraic action of the group $Z$ on $\glsfibre$ whose orbits are the orbits of $A$ in $\glsfibre$.  We have thus proven the following theorem.
	 \begin{thm}\label{thm:sregact}
	 Let $x\in\glsfibre$ be arbitrary and let $Z_{i}$ be the centralizer in $GL(i)$ of the Jordan form of $x_{i}$ (with eigenvalues in decreasing lexicographical order).  On $\glsfibre$ the orbits of the group $A$ are orbits of a free algebraic action of the connected abelian algebraic group $Z=\Gprod$.
	 	\end{thm}

We end this section with a result that will be of great use in section \ref{s:counting} where we count the number of $A$-orbits in the fibre $\glsfibre$. 

It turns out that the condition in Theorem \ref{thm:Aorb} that $\orbi\subset \fgl(i+1)^{reg}$ is superfluous.
\begin{thm}\label{thm:regorbits}
If $\orbi\subset\sol$ is a free $Z_{i}$-orbit, then $\orbi\subset\fgl(i+1)^{reg}$.

\end{thm}
\begin{proof}
Let $c=(c_{1}, c_{2},\cdots, c_{j}, c_{j+1},\cdots, c_{n})\in \C^{\frac{n(n+1)}{2}}$, with $c_{j}\in\C^{j}$ be given.  By Theorem \ref{thm:Hess}, there is a unique upper Hessenberg matrix $h\in\glsfibre$.   This implies that for any $j$ , $1\leq j\leq n-1$, there exists a $g_{j}\in GL(j)$ such that $(\Ad(g_{j})\cdot h)_{j+1}\in\Xi^{j}_{c_{j}, c_{j+1}}$ by Remark \ref{r:stupid}.  Thus, $\Ad(g_{j})\cdot h_{j+1}\in Z_{j}\cdot x_{a_{j}}=\orbj$ for some $x_{a_{j}}\in\Xi^{j}_{c_{j}, c_{j+1}}$.  But $h\in\fgl(n)^{sreg}$ and therefore $h_{j+1}$ is regular by part (a) of Proposition \ref{prop:sreg}, which implies that $\orbj\subset \fgl(j+1)^{reg}$.  Also, by part (b) of Proposition \ref{prop:sreg}, $Z_{j}$ acts freely on $\mathcal{O}_{a_{j}}^{j}$, as in the proof of the last statement of Theorem \ref{thm:Aorb}.  Thus, for any $j$, $1\leq j\leq n-1$, there exists a free $Z_{j}$-orbit in $\Xi^{j}_{c_{j}, c_{j+1}}$ consisting of regular elements of $\fgl(j+1)$.  

Now, let $\orbi\subset \sol$ be any free $Z_{i}$-orbit.  Now, we use the free $Z_{j}$-orbit $\mathcal{O}^{j}_{a_{j}}\subset \fgl(j+1)^{reg}$ as above for $1\leq j\leq i-1$ and $\orbi$ to construct a morphism $\Gammai: Z_{1}\times\cdots\times Z_{i}\to \glfibre\cap S$.  By Theorem \ref{thm:Aorb}, $Im\Gammai\subset \fgl(i+1)^{sreg}$.  Proposition \ref{prop:sreg} (a) then implies $Im\Gammai\subset\fgl(i+1)^{reg}$.   Since elements of $\orbi$ are conjugate to elements of $Im\Gammai$, $\orbi\subset \fgl(i+1)^{reg}$.  This completes the proof.  
\end{proof}




	 \subsection{$A$-orbits in $\glfibre\cap S$}\label{s:gengamma}
	  We now discuss how the construction in sections \ref{s:gammamap} and \ref{s:gammageo} can be generalized to describe $A$-orbits of dimension strictly less than ${n\choose 2}$ in the Zariski open subset of the fibre $\glfibre\cap S$.  In this case, it is more difficult to define the morphism $\Gamman$ that appears in equation (\ref{eq:gengamma}).  The problem is that it is not clear how to define a morphism $\orbi\to GL(i+1)$ which sends $x\to g_{i,i+1}(x)$ where $\Ad(g_{i,i+1}(x))\cdot x$ is in Jordan form (with eigenvalues in decreasing lexicographical order).   This is not difficult in the strongly regular case, as we are dealing with free $Z_{i}$-orbits $\orbi\simeq Z_{i}$ so that $g_{i, i+1}(x)$ can be defined as in equation (\ref{eq:gmatrix}).  The fortunate fact is that even for an orbit $\orbi\subset\sol$ of dimension strictly less than $i$, there exists a connected, Zariski closed subgroup $K_{i}\subset Z_{i}$ with $K_{i}$ acting freely on $\orbi\simeq K_{i}$.  Therefore, we can mimic what we did in equation (\ref{eq:gmatrix}).  

	  To prove this, we need to understand better the action of $Z_{i}$ on $\sol$.  As in section \ref{s:gammasum}, let $J=J_{\lambda_{1}}\oplus\cdots\oplus J_{\lambda_{r}}$ be the $i\times i$ cutoff of the matrix in (\ref{eq:bigmatrix}), where $J_{\lambda_{j}}\in\fgl(n_{j})$ is the Jordan block corresponding to eigenvalue $\lambda_{j}$.  We note since $J$ is regular, $Z_{i}$ is an abelian connected algebraic group, which is the product of groups $\prod_{j=1}^{r}Z_{ J_{\lambda_{j}}} $, where $Z_{ J_{\lambda_{j}}}$ denotes the centralizer of $J_{\lambda_{j}}$.  It is then easy to see that the action of $Z_{i}$ is the diagonal action of the product $\prod_{j=1}^{r}Z_{ J_{\lambda_{j}}}$ on the last column of $x\in \Xi^{i}_{c_{i},c_{i+1}}$ and the dual action on the last row of $x$ (see (\ref{eq:bigmatrix})).  In other words, $Z_{J_{\lambda_{j}}}$ acts only on the columns and rows of $x$ that contain the Jordan block $J_{\lambda_{j}}$ (see (\ref{eq:bigmatrix})).  This leads us to define an action of $Z_{J_{\lambda_{j}}}$ on $\Co^{2 n_{j}}$ 
	  \begin{equation}\label{eq:littleact}
	  z\cdot ([t_{1}, \cdots, t_{n_{j}}], [s_{1}, \cdots, s_{n_{j}}]^{T})=([t_{1}, \cdots, t_{n_{j}}] \cdot z^{-1}, z\cdot [s_{1}, \cdots, s_{n_{j}}]^{T}).
	  \end{equation}
	  For $x\in\sol$, let $\mathcal{O}$ be its $Z_{i}$-orbit, and let $\mathcal{O}_{j}\subset\Co^{2\, n_{j}}$ be the $Z_{ J_{\lambda_{j}}}$-orbit of $x[j]=([z_{j,1},\cdots, z_{j, n_{j}}], [y_{j,1}, \cdots, y_{j, n_{j}}])$ (where the coordinates for $x$ are as in (\ref{eq:bigmatrix})).  It follows directly from our above remarks that 
	\begin{equation}\label{eq:prodorb}
	\mathcal{O}\simeq\mathcal{O}_{1}\times\cdots\times\mathcal{O}_{r},
	\end{equation}
where the isomorphism is $Z_{i}$-equivariant.  Using this description of a $Z_{i}$-orbit $\mathcal{O}\subset\sol$, it is easy to describe the structure of the isotropy groups for the $Z_{i}$-action.  
   \begin{lem}\label{l:stab}
Let $x\in\Xi^{i}_{c_{i},c_{i+1}}$ and let $Stab(x)\subset Z_{i}$ be the isotropy group of $x$ under the action of $Z_{i}$ on $\Xi^{i}_{c_{i},c_{i+1}}$.  Then, up to reordering, 
\begin{equation}\label{eq:stab}
Stab(x)=\prod_{j=1}^{q} Z_{ J_{\lambda_{j}}}\times\prod_{j=q+1}^{r} U_{j},
\end{equation}
where $U_{j}\subset Z_{ J_{\lambda_{j}}}$ is a unipotent Zariski closed subgroup (possibly trivial) for some $q$, $0\leq q\leq r$.
\end{lem}
\begin{proof}
Suppose that $x\in\sol$ is given by (\ref{eq:bigmatrix}).  For ease of notation, we let $Z_{J_{\lambda_{k}}}=Z_{J_{k}}$.  
%
By equation (\ref{eq:prodorb}), to compute the stabilizer of $x$ we need only compute the stabilizers for each of the $Z_{J_{k}}$ orbits $\mathcal{O}_{k}=Z_{J_{k}}\cdot x[k]$, where $1\leq k\leq r$.  To compute the stabilizer of $x[k]$, suppose that $y_{k, i}\neq 0$, but for $i< l\leq n_{k}$, $y_{k,l}=0$.   We consider the matrix equation:
\begin{equation}\label{eq:toep}
A_{k}\cdot\underline{y_{k}}=\underline{y_{k}},
\end{equation}
where $A_{k}\in Z_{J_{k}}$ is an invertible upper triangular Toeplitz matrix and $\underline{y_{k}}\in\mathbb{C}^{n_{k}}$ is the column vector $\underline{y_{k}}=(y_{k,1},\cdots,\, y_{k,i}, \,0,\cdots, 0)^{T}$.  As $A_{k}$ is an upper triangular Toeplitz matrix, we see by considering the $i$th row in equation (\ref{eq:toep}) that $A_{k}$ is forced to be unipotent.  If on the other hand, all $y_{k, j}=0$ for $1\leq j\leq n_{k}$, we can argue similarly using the $z_{k, j}$ and the dual action. 

If $y_{k,l}=0\text{ for all } l$ and $z_{k,l}=0\text{ for all } l$,  then clearly the stabilizer of $x[k]$ is $Z_{J_{k}}$ itself.  Repeating this analysis for each $k$,  $1\leq k\leq r$ and after possibly reordering the Jordan blocks of $x_{i}$, we get the desired result. 
\end{proof}
We have an immediate corollary to the lemma which we stated before Theorem \ref{thm:Aorb} as Lemma \ref{l:conn}.
\begin{cor}\label{c:conn}
For any $x\in \Xi^{i}_{c_{i},c_{i+1}}$ $Stab(x)$ is connected.
\end{cor}
\begin{proof}
Upon reordering the eigevalues, we can always assume that $Stab(x)$ has the form given in (\ref{eq:stab}) in Lemma \ref{l:stab}.  This proves the result,  since unipotent algebraic groups are always connected and the groups $Z_{J_{\lambda_{j}}}$ are connected, since they are centralizers of regular elements in $\fgl(n_{j})$.\\
\end{proof}

We can now prove the structural theorem about the group $Z_{i}$ that lets us construct the morphism $\Gamman$ in the general case.  
	 
	\begin{thm}\label{thm:split}
	Let $x\in\sol$ and let $Stab(x)\subset Z_{i}$ denote the isotropy group of $x$ under the action of $Z_{i}$ on $\sol$.  Then as an algebraic group, 
	$$
	Z_{i}=Stab(x)\times K,
	$$
	where $K$ is a connected, Zariski closed algebraic subgroup of $Z_{i}$. 
		\end{thm}
	\begin{proof}
	For the purposes of this proof we denote by $H$ the group $Stab(x)$.  Without loss of generality, we assume $H$ is as given in (\ref{eq:stab}). Let $\fz_{i}=Lie(Z_{i})$ and let $\mathfrak{h}=Lie(H)$.  Now, by Lemma \ref{l:stab}, $\mathfrak{h}$
\begin{equation}\label{eq:liestab}
\mathfrak{h}=\bigoplus_{j=1}^{q}\mathfrak{z}_{J_{\lambda_{j}}}\oplus\bigoplus_{j=q+1}^{r}\mathfrak{n}_{j}, 
\end{equation}
where $\mathfrak{z}_{J_{\lambda_{j}}}$ is the Lie algebra of the abelian algebraic group $Z_{J_{\lambda_{j}}}$ and $\mathfrak{n}_{j}=Lie(U_{j})$ is a Lie subalgebra of $\mathfrak{n}^{+}(n_{j})$, the strictly upper triangular matrices in $\fgl(n_{j})$.   

The proof proceeds in two steps.  We first find an algebraic Lie subalgebra $\fk\subset\fz_{i}$ such that $\fz_{i}=\fh\oplus\fk$ as Lie algebras.   We then show that if $K\subset Z_{i}$ is the corresponding Zariski closed subgroup $Z_{i}=H\, K$ and $H\cap K=\{e\}$.  To find $\fk$, consider the abelian Lie algebra $\fz_{J_{\lambda_{j}}}$ for $q+1\leq j\leq r$.  Since $\fz_{J_{\lambda_{j}}}$ is abelian, it has a Jordan decomposition as a direct sum of Lie algebras $\fz_{J_{\lambda_{j}}}=\fz_{J_{\lambda_{j}}}^{ss}\oplus\fz_{J_{\lambda_{j}}}^{n}$, where $\fz_{J_{\lambda_{j}}}^{ss}$ are the semisimple elements of $\fz_{J_{\lambda_{j}}}$ and $\fz_{J_{\lambda_{j}}}^{n}$ are the nilpotent elements.   Now the Lie algebra $\fn_{j}$ in (\ref{eq:liestab}) is a subalgebra of $\fz_{J_{\lambda_{j}}}^{n}$.  Take $\widetilde{\fn_{j}}$ so that $\fz_{J_{\lambda_{j}}}^{n}=\fn_{j}\oplus \widetilde{\fn_{j}}$.  Let 
$$
\fm_{j}=\fz_{J_{\lambda_{j}}}^{ss}\oplus\widetilde{\fn_{j}}.
$$
Note that $\fm_{j}\oplus\fn_{j}=\fz_{J_{\lambda_{j}}}$.  We claim that $\fm_{j}$ is an algebraic subalgebra of $\fz_{J_{\lambda_{j}}}$.  Indeed, $\widetilde{\fn_{j}}$ is algebraic, since it is a nilpotent Lie algebra (see \cite{TY}, pg 383).   Let $\widetilde{N_{j}}$ be the corresponding algebraic subgroup.  Then $M_{j}=\Co^{\times}\times \widetilde{N_{j}}$ has $Lie(M_{j})=\mathfrak{m}_{j}$, as $\Ct$ is the semisimple part of group $Z_{J_{\lambda_{j}}}$ (see (\ref{eq:bigmatrix})).  We then take 
 $$
 \fk=\bigoplus_{j=q+1}^{r}\fm_{j}.
 $$
 This finishes the first step.  
 
 Let $K=\prod_{j=q+1}^{r} M_{j}$ be the Zariski closed, connected algebraic subgroup of $\prod_{j=q+1}^{r} Z_{J_{\lambda_{j}}}$ that corresponds to the algebraic Lie algebra $\fk$.  We now show that $Z_{i}=H\times K$.  By our choice of $K$, $H\cap K$ is finite.  But we also have, $H\cap K\subset \prod_{j=q+1}^{r} U_{j}$ and is thus unipotent (see (\ref{eq:stab})).  Since any unipotent group must be connected, we have $H\cap K=\{e\}$.  Now, it is clear that $Z_{i}=H\, K$, as $H\, K$ is a closed, connected subgroup of $Z_{i}$ of dimension $\dim Z_{i}$.    This completes the proof.  
	\end{proof}
	 
	 With Theorem \ref{thm:split} in hand, we can define the general $\Gamman$ morphism of (\ref{eq:gengamma}), as we did in the strongly regular case.  Now, suppose we are given $Z_{i}$-orbits in $\Xi^{i}_{c_{i},c_{i+1}}$,  $\mathcal{O}_{a_{i}}^{i}=K_{a_{i}}\cdot x_{a_{i}}\simeq K_{a_{i}}$ with $K_{a_{i}}$ as in Theorem \ref{thm:split} for $1\leq i\leq n-1$, and with $\mathcal{O}_{a_{i}}^{i}$ consisting of regular elements of $\fgl(i+1)$ for $1\leq i\leq n-2$.  We define a morphism
	\begin{equation*}
	\Gamma_{n}^{a_{1},\cdots, a_{n-1}}: K_{a_{1}}\times\cdots\times K_{a_{n-1}}\rightarrow \fgl(n)_{c}\cap S,
\end{equation*}
as in equation (\ref{eq:defn}). 

Propositions \ref{r:sense} and \ref{prop:Aacts}, Theorem \ref{thm:embedd}, and Remark \ref{r:different} from the strongly regular case remain valid in this case by simply replacing the groups $Z_{i}$ by the groups $K_{a_{i}}$.  We recall that the main ingredient in proving Theorem \ref{thm:embedd} is the fact that the group $Z_{i}$ acts freely on $\orbi$.  The analogue of Theorem \ref{thm:Aorb} remains valid in this, as it is easy to show 
$$
T_{y}(Im\Gamman)=V_{y},
$$
with $V_{y}$ as in (\ref{eq:dist}).

 We obtain at last Theorem \ref{thm:gengamma}.  
\begin{thm}
The image of the map $\Gamma_{n}^{a_{1},a_{2}, \cdots , a_{n-1}}$ is exactly one $A$-orbit in $\fgl(n)_{c}\cap S$.  Moreover, every $A$-orbit in $\fgl(n)_{c}\cap S$ is of the form $Im\Gamman$ for some choice of orbits $\orbi\subset\sol$, with $\orbi$ consisting of regular elements of $\fgl(i+1)$ for $1\leq i\leq n-2$.  
\end{thm}

	
	The following corollary of Thoerem \ref{thm:gengamma} is a generalization of Theorem 3.14 in \cite{KW1} to include elements that are not necessarily strongly regular.  
	 \begin{cor}
	 Let $x\in\glfibre\cap S$.  The $A$-orbit of $x$, $A\cdot x$ is a smooth, irreducible subvariety of $\fgl(n)$ that is isomorphic as an algebraic variety to a closed subgroup $K_{a_{1}}\times\cdots\times K_{a_{n-1}}$ of the connected algebraic group $\Gprod$.  
	 \end{cor}

	 \section{Counting $A$-orbits in $\glsfibre$ }\label{s:counting}
	Using Theorem \ref{thm:Aorb}, we can count the number of $A$-orbits in $\glsfibre$ for any $c\in\mathbb{C}^{\frac{n(n+1)}{2}}$ and explicitly describe the orbits.  From Theorem \ref{thm:Aorb} and Remark \ref{r:different} counting the number of $A$-orbits in $\glsfibre$ is equivalent to counting the number of $Z_{i}$-orbits in $\sol$ on which $Z_{i}$ acts freely.  We show in this section that the number of such orbits is directly related to the number of degeneracies in the roots of the monic polynomials $p_{c_{i}}(t)$ and $p_{c_{i+1}}(t)$ (see (\ref{eq:polyci})).  The study of this problem can be reduced to studying the structure of nilpotent solution varieties $\Xi^{i}_{0,0}$.  Thus, we begin our discussion by describing the $A$-orbit structure of the nilfibre $\glsnil$.

	\subsection{Nilpotent solution varieties and $A$-orbits in the nilfibre} 
In this section, we study strongly regular matrices in the fibre $\glnil$. By definition $x\in\glnil$ if and only if $x_{i}\in\fgl(i)$ is nilpotent for all $i$.  Such matrices have been studied by \cite{Ov} and \cite{PS}.   

	We restate Definition \ref{def:sol} of the solution variety $\sol$ in this case.  
Elements of $\fgl(i+1)$ of the form
\begin{equation}\label{eq:nilform}
X=\left[\begin{array}{cc}
\begin{array}{cccc}
0& 1 & \cdots & 0\\
0& 0&\ddots &\vdots\\
\vdots&\;& \ddots & 1\\
0& \cdots &\cdots &0\\
\end{array} & \begin{array}{c}
y_{1}\\
\vdots\\
\vdots\\
y_{i}\end{array}\\
\begin{array}{cccc}
z_{1}&\cdots&\cdots&z_{i}
\end{array} & w
\end{array}\right]
\end{equation}
 which are nilpotent define the nilpotent solution variety at level $i$, which we denote by $\nsol$.  In this case, it is easy to write down elements in $\nsol$.  For example, we can take all of the $z_{j}$, $y_{j}$, and $w$ to be $0$.   However, such an element is not regular, and so cannot be used to construct a $\Gamman$ mapping that gives rise to a strongly regular orbit in $\fgl_{0}^{sreg}(n)$.  To describe $A$-orbits in $\glsnil$, we focus our attention on free $Z_{i}$-orbits in $\nsol$, (see Theorem \ref{thm:Aorb}).  To find such orbits, we need to compute the characteristic polynomial of $X$.  
\begin{prop}\label{prop:det2}
The characteristic polynomial of the matrix in (\ref{eq:nilform}) is 
\begin{equation}\label{eq:nilpoly}
\begin{array}{c}
\det(X-t)=(-1)^{i}\left[-t^{i+1}+wt^{i}+\sum_{l=0}^{i-1}\sum_{j=1}^{i-l}z_{j}y_{j+l} t^{i-1-l}\right]. \end{array}
 \end{equation}
\end{prop}

\begin{proof}
We compute the characteristic polynomial for the matrix in (\ref{eq:nilform}) using the Schur complement formula for the determinant (see \cite{HJ}, pgs 21-22).  In the notation of that reference $\alpha=\{1,\cdots, n-1\}$ and $\alpha^{\prime}=\{n\}$.  Let $J=X_{i}$ denote the principal nilpotent Jordan block.  Then the Schur complement formula in \cite{HJ} gives
\begin{equation}\label{eq:Schur}
\det(X-t)=\det(J-t)\,(w-t)-\underline{z}\; adj(J-t)\;\underline{y},
\end{equation}
where $adj(J-t)\in\fgl(i)$ denotes the classical adjoint matrix, $\underline{z}=[z_{1},\cdots, z_{i}]$ is a row vector, and $\underline{y}=[y_{1},\cdots , y_{i}]^{T}$ is a column vector.  We easily compute that $\det(J-t)=(-1)^{i} t^{i}$.  It is not difficult to see that 
$$
adj(J-t)=(-1)^{i-1}\left[\begin{array}{cccccc} t^{i-1} & t^{i-2} &\cdots & \cdots & t &1\\
												  0        & t^{i-1} & t^{i-2} &\cdots & \cdots & t\\
												  \vdots& 0 & t^{i-1} & \ddots & & \vdots\\
												  & &0 & \ddots & \ddots &\vdots \\
												  \vdots & & &\ddots & \ddots & t^{i-2}\\
												  0 &\cdots & &\cdots & 0 & t^{i-1}\end{array}\right]_{\mbox{\large .}}
$$
Now, we compute that the coefficient of $t^{i-1-l}$ for $0\leq l\leq i-1$ in the product $\underline{z}\; adj(J-t)\;\underline{y}^{T}$ is 
\begin{equation}\label{eq:coeffj}
(-1)^{i-1}\sum_{j=1} ^{i-l} z_{j} y_{j+l}.
\end{equation}
Summing up the terms in (\ref{eq:coeffj}) for $0\leq l\leq i-1$ and using equation (\ref{eq:Schur}), we obtain the polynomial in (\ref{eq:nilpoly}).  
\end{proof}
   
    For the matrix in (\ref{eq:nilform}) to be nilpotent, we require that all of the coefficients of the polynomial in (\ref{eq:nilpoly}) (excluding the leading coefficient) vanish. 

\begin{equation}\label{eq:nilcond}
\begin{array}{c}
z_{1}y_{i}=0\\
\\
z_{1} y_{i-1}+ z_{2} y_{i}=0\\

\vdots \\

z_{1}y_{1}+\cdots + z_{1} y_{i}=0\end{array}
\end{equation}

We claim that $\Xi_{0,0}^{i}$ has exactly two free $Z_{i}$-orbits.  These correspond to choosing either $z_{1}\in\Ct, \,y_{i}=0 $, or $y_{i}\in\Ct, \, z_{1}=0$ in the first equation of (\ref{eq:nilcond}).  We claim that any point in $\nsol$ with $z_{1}\neq 0$ is in
\begin{equation}\label{eq:lower}
  \mathcal{O}_{L}^{i}=\left[\begin{array}{cc}
\begin{array}{cccc}
0& 1 & \cdots & 0\\
0& 0&\ddots &\vdots\\
\vdots&\;& \ddots & 1\\
0& \cdots &\cdots &0\\
\end{array} & \begin{array}{c}
0\\
\vdots\\
\vdots\\0
\end{array}\\
\begin{array}{cccc}
z_{1}&\cdots&\cdots&z_{i}
\end{array} & 0
\end{array}\right]_{\mbox{\large ,}}
\end{equation}
with $ z_{j}\in\mathbb{C}\, , 2\leq j\leq i$.
Any point in $\nsol$ with $y_{i}\in\Ct$ is in
\begin{equation}\label{eq:upper}
\mathcal{O}_{U}^{i}=\left[\begin{array}{cc}
\begin{array}{cccc}
0& 1 & \cdots & 0\\
0& 0&\ddots &\vdots\\
\vdots&\;& \ddots & 1\\
0& \cdots &\cdots &0\\
\end{array} & \begin{array}{c}
y_{1}\\
\vdots\\
\vdots\\y_{i}
\end{array}\\
\begin{array}{cccc}
0&\cdots&\cdots&0
\end{array} & 0
\end{array}\right]_{\mbox{\large ,}}
\end{equation}
with $ y_{j}\in\mathbb{C}\, , 1\leq j\leq i-1$.  To verify this claim, note that if $z_{1}\neq 0$ and $y_{i}=0$, then $y_{1}=0,\, y_{2}=0,\cdots , y_{i-1}=0$ by successive use of equations (\ref{eq:nilcond}).  The case $y_{i}\neq 0$, $z_{1}=0$ is similar.   An easy computation in linear algebra, as in the proof of Lemma \ref{l:stab} gives that $Z_{i}$ acts freely on $\orbu$ and $\orbl$.  We think of $\orbu$ as the ``upper orbit " in $\nsol$ and $\orbl$ as the ``lower orbit".  Both orbits consist of regular elements of $\fgl(i+1)$ by Theorem \ref{thm:regorbits}.

Now, suppose that both $z_{1}=0=y_{i}$ in (\ref{eq:nilcond}).  It is easy to see that such an element has a non-trivial isotropy group in $Z_{i}$ containing the one dimensional subgroup of matrices
$$
\left[\begin{array}{cccc}
1& 0 & \cdots & c\\
0& 1&\ddots &\vdots\\
\vdots&\text{ }& \ddots & 0\\\
0& \cdots &\cdots &1\\
\end{array}\right ]_{\mbox{\large ,}}
$$
with $c\in\Ct$.  It does not belong to a $Z_{i}$-orbit of dimension $i$.  

Thus, to analyze $\fgl_{0}^{sreg}(n)$, we consider only the $Z_{i}$-orbits $ \mathcal{O}_{U}^{i}\, , \mathcal{O}_{L}^{i}$.  Using the orbits $\mathcal{O}_{U}^{i}\, , \mathcal{O}_{L}^{i}$, we can construct $2^{n-1}$ morphisms of the form $\Gamma_{n}^{a_{1},a_{2}, \cdots , a_{n-1}}$ where $a_{i}=\mathcal{O}_{U}^{i}\, , \mathcal{O}_{L}^{i}$ for $1\leq i\leq n-1$.  The following result follows immediately from Theorems \ref{thm:Aorb} and \ref{thm:sregact} and Remark \ref{r:different}.


\begin{thm}\label{thm:nil}
The nilfibre $\glsnil$ contains $2^{n-1}$ $A$-orbits.  On $\glsnil$ the orbits of $A$ are orbits of a free action of the algebraic group $(\Ct)^{n-1}\times\C^{{n\choose 2}-n+1}$.
\end{thm}

The nilfibre has much more structure than Theorem \ref{thm:nil} indicates.   We can see this additional structure by considering an example of an $A$-orbit given as the image of  a morphism $\Gamman$ with $\orbi=\orbu, \, \orbl$ and its closure.   Closure here means either closure in the Zariski topology in $\fgl(n)$ or in the Euclidean topology, since $A$-orbits are constructible sets these two different types of closure agree (see Theorem 3.7 in \cite{KW1}).  For ease of notation, we will abbreviate from now on $\orbl=L$, $\orbu=U$.  


\begin{exam}\label{ex:nil4}
Let us take our $A$-orbit in $\fgl(4)_{0}^{sreg}$ to be the image of $\Gamma_{4}^{a_{1}, a_{2}, a_{3}}$ with $a_{1}=L, a_{2}=L, a_{3}=U$.  For coordinates, let us take for $\mathcal{O}^{1}_{L}$, $z_{1}\in\Ct$, for $\mathcal{O}^{2}_{L}$, $z_{2}\in\Ct,\, z_{3}\in\Co$, and for $\mathcal{O}^{3}_{U}$, $y_{1},\, y_{2}\in\Co,\, y_{3}\in\Ct$.  In these coordinates, we compute that $Im\Gamma^{L, L, U}$ is

\begin{equation}
Im\Gamma^{L, L, U}=\left[\begin{array}{cccc}
0 & 0 & 0 & \frac{y_{3}}{z_{1}z_{2}}\\
z_{1} & 0 & 0 & \frac{y_{2}}{z_{2}}- \frac{y_{3}z_{3}}{z_{2}^{2}}\\
z_{1}\, z_{3} & z_{2} & 0 & y_{1}\\
0 & 0 & 0& 0\end{array}\right]_{\mbox{\large .}}
\end{equation}
Since for $x\in\glsfibre$, $\overline{A\cdot x}$ is an irreducible variety of dimension ${n\choose 2}$ by Theorem 3.12 in \cite{KW1}, we compute the closure
\begin{equation}\label{eq:nilexample}
\overline{Im\Gamma^{L, L, U}}=\left[\begin{array}{cccc}
0 & 0 & 0 & a_{1}\\
a_{2} & 0 & 0 & a_{3}\\
a_{4} & a_{5} & 0 & a_{6}\\
0 & 0 & 0 &0\end{array}\right]_{\mbox{\large ,}}
\end{equation}
with $a_{i}\in \Co$ for $1\leq i\leq 6$.  $\overline{Im\Gamma^{L, L, U}}$ is a nilradical of a Borel subalgebra that contains the standard Cartan subalgebra of diagonal matrices in $\fgl(4)$.  The easiest way to see this is to note that the strictly lower triangular matrices in $\fgl(4)$ are conjugate to $\overline{Im\Gamma^{L, L, U}}$  by the permutation $\tau=(1432)$.  
\end{exam}

This example illustrates that the $A$-orbits in $\glsnil$ are essentially parameterized by prescribing whether or not the the $i\times i$ cutoff of an element $x\in\glnil$ has zeroes in its $i$th column or zeroes in its $i$th row.  This is because for an $x\in\glnil$ to be in the image of a morphism $\Gamman$ with $a_{i}=L,\, U$, the $i$th row or the $i$th column of $x_{i}$ must entirely consist of zeroes for each $i$ by Proposition \ref{r:sense}. 

Contrast this with the following example of a matrix $x\in\glnil$ each of whose cutoffs is regular, but that is not strongly regular.
\begin{exam}\label{ex:nsreg}
Consider $x\in\fgl(4)_{0}$
\begin{equation}
x=\left[ \begin{array}{cccc}
	0 & 0& 0 & 0\\
	1 & 0 & 0 & x_{2}\\
	0 & 1 & 0 & x_{3}\\
	y_{1}& 0 & 0 & 0\end{array}\right ]_{\mbox{\large ,}}
	\end{equation}
	where $x_{2}\in\mathbb{C}^{\times},\, y_{1}\in\mathbb{C}^{\times}\text{, and } x_{3}\in\mathbb{C} $.   Note that both the $4$th column and row of this matrix have non-zero entries. Thus, this matrix cannot be in the image of a morphism $\Gamman$ with $a_{i}=L, U$ and is not strongly regular.  However, one can easily check that each cutoff of this matrix is regular so that $x\in \fgl(4)_{0}\cap S$.  Thus, $\fgl(4)_{0}^{sreg}$ is a proper subset of $\fgl(4)_{0}\cap S$.  (One can also see that this matrix is not strongly regular directly by observing that $\fz_{\fgl(3)}(x_{3})\cap \fz_{\fgl(4)}(x)\neq 0$. ) 	
\end{exam}

Example \ref{ex:nil4} demonstrates that although the $A$-orbits $Im\Gamman$ may be complicated, their closures are relatively simple.  In this example, the closure is a nilradical of a Borel subalgebra that contains the standard Cartan subalgebra of diagonal matrices in $\fgl(n)$.  This is in fact the case in general.  

\begin{thm}\label{thm:nilradical}
Let $x\in\glsnil$ and let $A\cdot x$ denote that $A$-orbit of $x$.  Then $\overline{A\cdot x}$ is a nilradical of a Borel subalgebra in $\fgl(n)$ that contains the standard Cartan subalgebra of diagonal matrices. More explicitly, if the $A$-orbit is given by $\Gamma_{n}^{a_{1},a_{2}, \cdots , a_{n-1}}$ where $a_{i}=U\, \text{ or } L$ for $1\leq i\leq n-1$, then $\overline{A\cdot x} $ is the set of matrices of the following form
$$
\nilrad:=\left\{x: x_{i+1}=\left [\begin{array}{cc} x_{i} & \begin{array}{c}  b_{1} \\
\vdots\\
b_{i}\end{array}\\
0 & 0 \end{array}\right  ]\right\}_{\mbox{\large ,}}
$$ 
with $b_{j}\in\mathbb{C}$ if $a_{i}=U$, or if $a_{i}=L$ 
$$
\nilrad:=\left\{x: x_{i+1}=\left [\begin{array}{cc} x_{i} & 0\\
\begin{array}{ccc} b_{1}&\cdots & b_{i}\end{array}& 0 \end{array}\right  ]\right\}_{\mbox{\large ,}}
$$
with $b_{j}\in\mathbb{C}$.  
\end{thm}
\begin{proof}
Let $x\in\glsnil$.  By Gerstenhaber's Theorem \cite{Ger}, it suffices to show the second statement of the theorem.  Then $\overline{A\cdot x}$ is a linear space consisting of nilpotent matrices of dimension ${n\choose 2}$, which is clearly normalized by the diagonal matrices in $\fgl(n)$.  

Suppose that $A\cdot x=Im\Gamman$ with $a_{i}=U,\, L$.  Then it is easy to see $A\cdot x\subset \nilrad$ by the definition of the morphism $\Gamman$ in section \ref{s:gammamap}.  By Theorem 3.12 in \cite{KW1}, $A\cdot x$ is an irreducible variety of dimension ${n\choose 2}$.  Thus, $\overline{A\cdot x}\subset \nilrad$ is an irreducible, closed subvariety of dimension ${n\choose 2}=\dim\nilrad$, and therefore $\overline{A\cdot x}=\nilrad$.


\end{proof}
\begin{rem}
 The strictly lower triangular matrices $\fn^{-}$ is the closure of the $A$-orbit $\Gamma_{n}^{L, \cdots, L}$, and the strictly upper triangular matrices $ \fn^{+}$ is the closure of the $A$-orbit $\Gamma_{n}^{U,\cdots, U}$.  

\end{rem}  

By Theorem \ref{thm:nilradical}, the $A$-orbits in $\glsnil$ give rise to $2^{n-1}$ Borel subalgebras of $\fgl(n)$ that contain the diagonal matrices.  Moreover, each of the nilradicals $\nilrad$ is conjugate to the strictly lower triangular matrices by a unique permutation in $\mathcal{S}_{n}$, the symmetric group on $n$ letters.  The $A$-orbits in $\glsnil$ thus determine $2^{n-1}$ permutations.  We now describe these permutations.   
\begin{thm}\label{thm:perms}
Let $\fn^{-}$ denote the strictly lower triangular matrices in $\fgl(n)$ and let $\nilrad$ be as in Theorem \ref{thm:nilradical}.  Then $\nilrad$ is obtained from $\fn^{-}$ by conjugating by a permutation $\sigma=\tau_{1}\,\tau_{2}\cdots \tau_{n-1}$ where $\tau_{i}\in\mathcal{S}_{i+1}$ is either the long element $w_{i, 0}$ of $\mathcal{S}_{i+1}$ or the identity permutation, $id_{i}$.  The $\tau_{i}$ are determined by the values of $a_{i}$ as follows.  Let $a_{n}=L$.   Starting with $i=n-1$, we compare $a_{i}, \, a_{i+1}$.  If $a_{i}=a_{i+1}$, then $\tau_{i}=id_{i}$, but if $a_{i}\neq a_{i+1}$, then $\tau_{i}=w_{0,i}$. 

The same procedure beginning with $a_{n}=U$ produces a permutation that conjugates the strictly upper triangular matrices $\fn^{+}$ into $\nilrad$.  
\end{thm}
Before proving Theorem \ref{thm:perms}, let us see it in action in Example \ref{ex:nil4}.  In that case the nilradical in equation (\ref{eq:nilexample}) is $\fn_{L, \, L,\, U}$.  Thus, according to Theorem \ref{thm:perms}, $\sigma=(13)(14)(23)$, the product of the long elements for $\mathcal{S}_{3}$ and $\mathcal{S}_{4}$.  Notice that $\sigma=(1432)$, which is precisely the permutation that we observed conjugates the strictly lower triangular matrices in $\fgl(4)$ into $\fn_{L,\, L,\, U}$ in Example \ref{ex:nil4}.  

We now prove Theorem \ref{thm:perms}.  In the proof, we will make use of the following notation.  Let $\pi_{i}: \fgl(n)\to \fgl(i)$ be the projection $\pi_{i}(x)=x_{i}$.  For any subset $S\subset \fgl(n)$ we will denote by $S_{i}$ the image $\pi_{i}(S)$.  
\begin{proof}
Suppose that $L=a_{n}=a_{n-1}=\cdots= a_{i+1}$, but $a_{i}=U$.  Conjugating $\fn^{-}$ by $\tau_{i}=w_{0,i}$ produces the nilradical $\Ad(\tau_{i})\cdot\fn^{-}$ with $(\Ad(\tau_{i})\cdot\fn^{-})_{i+1}=\fn^{+}_{i+1}$.  Thus, $(\nilrad)_{i+1}$ and $(\Ad(\tau_{i})\cdot \fn^{-})_{i+1}$ now have the same $i+1$ columns.  We also note that the components of $\Ad(\tau_{i})\cdot\fn^{-}$ and $\nilrad$ in $\fgl(i+1)^{\perp}$ also agree, as $\tau_{i}$ permutes the strictly lower triangular entries of the rows below the $i+1$th row of $\fn^{-}$ amongst themselves.  Now, we start the procedure again with $(\Ad(\tau_{i})\cdot \fn^{-})_{i+1}$ and $a_{i}=U$ use induction. We note that conjugating $\Ad(\tau_{i})\cdot \fn^{-}$ by a permutation in $\mathcal{S}_{k}$ with $k\leq i+1$ leaves the component of $\Ad(\tau_{i})\cdot \fn^{-}$ in $\fgl(i+1)^{\perp}$ unchanged.  This proves the theorem.  
\end{proof}
\begin{rem}\label{r:ps}
There is a related result in recent work of Parlett and Strang.  See Lemma 1 in \cite{PS}, pg 1736.  
\end{rem}
\subsection{General solution varieties $\sol$ and counting $A$-orbits in $\glsfibre$}
  Now, we use our understanding of the nilpotent case to count $A$-orbits in the general case.  Recall the definition of the solution variety $\sol$ in section \ref{s:gammasum}.   
We also recall some notation.  Given $c\in\Co^{\frac{n(n+1)}{2}}$, we write $c=(c_{1},\cdots, c_{i},\cdots, c_{n})$ with $c_{i}=(z_{1},\cdots, z_{i})\in\Co^{i}$ and define a corresponding monic polynomial $\pci$ with coefficients given by $c_{i}$ (see (\ref{eq:polyci})).  Recall also that $J=J_{\lambda_{1}}\oplus\cdots\oplus J_{\lambda_{r}}$, $J_{\lambda_{k}}\in\fgl(n_{k})$, denotes the regular Jordan form that is the $i\times i$ cutoff of the matrix in (\ref{eq:bigmatrix}).  We now describe the $Z_{i}$-orbit structure of the variety $\sol$ for any $c_{i}\in\C^{i}$ and $c_{i+1}\in\C^{i+1}$.   

  As in the nilpotent case, to understand $\sol$ we must compute the characteristic polynomial of the matrix in (\ref{eq:bigmatrix}).
\begin{prop}\label{prop:charpoly}
The characteristic polynomial of the matrix in (\ref{eq:bigmatrix}) is
\begin{equation}\label{eq:solve}
(w-t)\prod_{k=1}^{r}(\lambda_{k}-t)^{n_{k}}
+\sum_{j=1}^{r}\left [(-1)^{n_{j}}\prod_{k=1,k\neq j}^{r}(\lambda_{k}-t)^{n_{k}}\sum_{l=0}^{n_{j}-1}\sum_{j'=1}^{n_{j}-l}z_{j,j'}y_{j,j'+l}(t-\lambda_{j})^{n_{j}-1-l}\right ]_{\mbox{\large .}}
\end{equation}
\end{prop}

 The proof of this proposition reduces to the case where $J$ is a single Jordan block of eigenvalue $\lambda$.   The case of a single Jordan block follows easily from the nilpotent case in Proposition \ref{prop:det2} by a simple change of variables.  

We need to understand the conditions that $z_{i,j}$, $y_{i,j}$, and $w$ must satisfy so that polynomial in (\ref{eq:solve}) is equal to the monic polynomial $\pcp$.  $w$ is easily determined by considering the trace of the matrix in (\ref{eq:bigmatrix}).  The values of the $z_{i,j}$ and the $y_{i,j}$ are directly related to the number of roots in common between the polynomials $\pci$ and $\pcp$.   Suppose that the polynomials $\pci$ and $\pcp$ have $j$ roots in common, where $1\leq j\leq r$.  Then we claim that $\sol$ has precisely $2^{j}$ free $Z_{i}$-orbits.  Consider the Jordan block corresponding to the eigenvalue $\lambda_{k}$.  First, suppose that $\lambda_{k}$ is a root of $\pcp$.  Then Proposition \ref{prop:charpoly} implies 
\begin{equation}\label{eq:overlap}
z_{k,1} y_{k, n_{k}}=0.
\end{equation}
However, if $\lambda_{k}$ is not a root of $\pcp$, then Proposition \ref{prop:charpoly} gives 
\begin{equation}\label{eq:disjoint}
z_{k,1} y_{k, n_{k}}\in\Ct.
\end{equation}

As in the nilpotent case, (\ref{eq:overlap}) gives rise to two separate cases.  
\begin{equation}\label{eq:case1}
z_{k,1}\in\Ct,\; y_{k, n_{k}}=0
\end{equation}
and 
\begin{equation}\label{eq:case2}
y_{k,n_{k}}\in\Ct, \; z_{k,1}=0.
\end{equation}

In the case of (\ref{eq:case1}), we can argue using (\ref{eq:solve}) that the coordinates $y_{k,i}$ for $1\leq i\leq n_{k}$ can be solved uniquely as regular functions of $z_{k,1}\in\Ct, \, z_{k,2},\cdots, z_{k, n_{k}}\in\Co$.  And in the case of (\ref{eq:case2}), we can solve for $z_{k,i}$ as regular functions of $y_{k, n_{k}}\in\Ct$ and $y_{k,i}\in\Co$, $1\leq i\leq n_{k}-1$.  In the case of (\ref{eq:disjoint}), we can take either the $z_{k,i}$ as coordinates that determine the $y_{k,i}$ or vice versa.  For concreteness, we take $y_{k,i}=p_{i}(z_{k,1},\cdots, z_{k, n_{k}})$ to be regular functions of $z_{k,1}\in\Ct, \, z_{k,2},\cdots, z_{k, n_{k}}\in \Co$.  

\begin{rem}\label{r:sols}
The solutions in the cases of (\ref{eq:overlap}) and (\ref{eq:disjoint}) are obtained by setting the derivatives of the polynomial in (\ref{eq:solve}) up to order $n_{p}-1$ evaluated at $\lambda_{p}$ equal to the corresponding derivatives of the polynomial $\pcp$ evaluated at $\lambda_{p}$ for $1\leq p\leq r$.  This produces $r$ systems of linear equations.  Each system involves only the coordinates $z_{p, k}$ and $y_{p, k}$ from the $p$th Jordan block.  This follows directly from the fact that the eigenvalues $\lambda_{s}$ are all distinct.    Each system can then be solved inductively using the fact that the coefficient of $(-1)^{n_{p}} (t-\lambda_{p})^{q}\prod_{k=1,\,k\neq p}^{r}(\lambda_{k}-t)^{n_{r}}$ is given by the $n-q\,th$ row of the matrix product
\begin{equation}\label{eq:matrixproduct}
\left[\begin{array}{cccc}
z_{p,1}& z_{p,2} & \cdots & z_{p,n_{p}}\\
0&z_{p,1}&\ddots &\vdots\\
\vdots&\text{ }& \ddots & z_{p,2}\\
0& \cdots & 0 &z_{p,1}\\
\end{array}\right ]\cdot\left [\begin{array}{c} y_{p,1}\\
\vdots\\ 
\vdots\\
y_{p,n_{p}}\end{array}\right ]_{\mbox{\large .}}
\end{equation}
\end{rem}

Recall that $Z_{i}$ is the direct product $Z_{i}=Z_{J_{\lambda_{1}}}\times\cdots\times Z_{J_{\lambda_{r}}}$, with $Z_{J_{\lambda_{s}}}$ the centralizer of $J_{\lambda_{s}}$.  The adjoint action of $Z_{i}$ on $\sol$ is a diagonal action where $Z_{J_{\lambda_{s}}}$ acts only on the columns and rows of an $x\in\sol$ containing $J_{\lambda_{s}}$.  This observation allowed us to decompose a $Z_{i}$-orbit $\mathcal{O}$ into the product of $ Z_{J_{\lambda_{k}}}$-orbits, $\mathcal{O}_{k}\subset\Co^{2 n_{k}}$ as in equation (\ref{eq:prodorb}), which we restate here for the convenience of the reader.
$$
\mathcal{O}\simeq\mathcal{O}_{1}\times \cdots\times\mathcal{O}_{k},
$$
where the isomorphism is $Z_{i}$-equivariant and $Z_{J_{\lambda_{k}}}$ acts on $\mathcal{O}_{k}$ as in equation (\ref{eq:littleact}).  If $\lambda_{k}$ is a root of $p_{c_{i+1}}(t)$, then (\ref{eq:overlap}) gives rise to two free $Z_{J_{\lambda_{k}}}$-orbits, an ``upper" orbit $\mathcal{O}_{U}^{k}$ in the case of (\ref{eq:case2}) and a ``lower" orbit $\mathcal{O}_{L}^{k}$ in the case of (\ref{eq:case1}).  This is proved similarly to the nilpotent case.  If on the other hand, $\lambda_{k}$ is not a root of $p_{c_{i+1}}(t)$, and we have (\ref{eq:disjoint}), then the vector 
\begin{equation}\label{eq:ugvector}
([z_{k,1},\cdots, z_{k, n_{k}}], [p_{1}(z_{k,1},\cdots, z_{k, n_{k}}),\cdots, p_{k}(z_{k,1},\cdots, z_{k, n_{k}})]^{T})\in\Co^{2 n_{k}}
\end{equation}
is a free $Z_{J_{k}}$-orbit under the action of $Z_{J_{k}}$ defined in (\ref{eq:littleact}).  Thus, using the orbits $\mathcal{O}_{U}^{k}$ and $\mathcal{O}_{L}^{k}$ for $1\leq k\leq j$, we can construct $2^{j}$ free $Z_{i}$-orbits in $\sol$ by (\ref{eq:prodorb}).  



Now, using Theorem \ref{thm:regorbits}, we can construct $2^{\sum_{i=1}^{n-1} j_{i}}$ $\Gamman$ morphisms into $\glsfibre$ where $j_{i}$ is the number of roots in common to the monic polynomials $\pci$ and $\pcp$.  The following result follows immediately from Theorem \ref{thm:Aorb} and Theorem \ref{thm:sregact} and Remark \ref{r:different}.  

\begin{thm}\label{thm:general}
Let $c=(c_{1}, c_{2},\cdots ,c_{i}, c_{i+1},\cdots , c_{n})\in\mathbb{C}^{\frac{n(n+1)}{2}}$.  Suppose there are $0\leq j_{i}\leq i$ roots in common between the monic polynomials $\pci$ and $\pcp$.  Then the number of $A$-orbits in $\glsfibre$ is exactly $2^{\sum_{i=1}^{n-1} j_{i}}.$  Further, on $\glsfibre$ the orbits of $A$ are the orbits of a free algebraic action of the commutative, connected algebraic group $Z=\Gprod$ on $\glsfibre$. 
\end{thm}

\begin{rem}
A similar result is obtained in recent work of Bielawski and Pidstrygach \cite{BP}.  See Remark \ref{r:bp} in the introduction.
\end{rem}

Theorem \ref{thm:general} lets us identify exactly where the action of the group $A$ is transitive on $\glsfibre$.  Let $\Theta_{n}$ be the set of $c\in\mathbb{C}^{\frac{n(n+1)}{2}}$ such that the monic polynomials $\pci$ and $\pcp$ have no roots in common.   From Remark 2.16 in \cite{KW1}, it follows that $\Theta_{n}\subset\Co^{\frac{n(n+1)}{2}}$ is Zariski principal open. 

\begin{cor}\label{c:generic}
The action of $A$ is transitive on $\glsfibre$ if and only if $c\in\Theta_{n}$. 
\end{cor}
\begin{rem}\label{r:willsee}
We will see in the next section that for $c\in\Theta_{n}$, $\glsfibre=\glfibre$.  Thus, for $c\in \Theta_{n}$ the fibre $\glfibre$ consists entirely of strongly regular elements.  
\end{rem}

Corollary \ref{c:generic} allows us to enlarge the set of generic matrices $\fgl(n)_{\Omega}$ studied by Kostant and Wallach. 
\subsection{The new set of generic matrices $\fgltheta$}\label{s:theta}
We can expand the set of matrices $\fglomega$ studied by Kostant and Wallach by relaxing the condition that each cutoff is regular semisimple.  More precisely, let $\sigma(x_{i})$ denote the spectrum of $x_{i}\in\fgl(i)$, where $x_{i}$ is viewed as an element of $\fgl(i)$.  We define a Zariski open subset of elements of $\fgl(n)$ by 
$$
\fgltheta=\{x\in\fgl(n) | \; \sigma(x_{i-1})\cap\sigma(x_{i})=\emptyset,\, 2\leq i\leq n\}.
$$
Clearly, $\fgltheta=\bigcup_{c\in\Theta_{n}}\glfibre$.  

\begin{thm}
The elements of $\fgltheta$ are strongly regular and therefore $\glsfibre=\glfibre$ for $c\in\Theta_{n}$.  Moreover, $\fgltheta$ is the maximal subset of $\fgl(n)$ for which the action of $A$ is transitive on the fibres of $\Phi$.
\end{thm}
\begin{proof}
If $\pci$ and $\pcp$ are relatively prime polynomials, then we claim $\sol$ is exactly one free $Z_{i}$-orbit.  Indeed, in this case we only have the conditions (\ref{eq:disjoint}) for $1\leq k\leq r$.  Thus, we can apply our observation in (\ref{eq:ugvector}) to see that $\sol$ is one free $Z_{i}$-orbit and hence consists of regular elements of $\fgl(i+1)$ by Theorem \ref{thm:regorbits}.  Given $x\in\glfibre$ with $c\in\Theta_{n}$, we claim that $x\in Im\Gamman$ with $a_{i}=\sol$ for $1\leq i\leq n-1$.  Indeed, $x_{2}\in\Xi^{1}_{c_{1},c_{2}}$ and is therefore regular.  Thus, by Remark \ref{r:stupid}, there exists a $g_{2}\in GL(2)$ such that $(\Ad(g_{2})\cdot x)_{3}=(\Ad(g_{2})\cdot x_{3})\in \Xi^{2}_{c_{2}, c_{3}}$.  Now, suppose $x_{i+1}\in \Ad(GL(i))\cdot \sol$.  Thus, $x_{i+1}\in\fgl(i+1)$ is regular and Remark \ref{r:stupid} provides a $g_{i+1}\in GL(i+1)$ such that $(\Ad(g_{i+1})\cdot x)_{i+2}=\Ad(g_{i+1})\cdot x_{i+2}\in \Xi^{i+1}_{c_{i+1}, c_{i+2}}$.  By induction, $x_{j+1}\in \Ad(GL(j))\cdot\Xi^{j}_{c_{j}, c_{j+1}}$ for any $j$, $1\leq j\leq n-1$.  Proposition \ref{r:sense} implies that $x\in Im\Gamman$.  Thus, by Theorem \ref{thm:Aorb}, $\fgltheta\subset\fgl(n)^{sreg}$.  The rest of the Theorem follows from Corollary \ref{c:generic}. 
\end{proof} 

\begin{rem}
For a matrix $x\in \glfibre$ where $c\in\Theta_{n}$, its strictly upper triangular part is determined by its strictly lower triangular part.  This follows from the definition of the morphisms $\Gamman$ and the fact that all of the $y_{k,i}$ can be solved uniquely as regular functions of the $z_{k,i}$ for $1\leq i\leq n_{k}$,  $1\leq k\leq r$.
\end{rem}

The fact that elements of $\fgltheta$ are strongly regular gives us the following corollary.
\begin{cor}
Let $x\in\fgltheta$.  Then $x_{i}\in\fgl(i)$ is regular for all $i$.  
\end{cor}

Using Corollary \ref{c:generic} and Theorem \ref{thm:general}, we get a direct generalization of Theorem 3.23 in \cite{KW1} for the case of $\Theta_{n}$.
\begin{cor}
For $c\in\Theta_{n}\subset \Co^{\frac{n(n+1)}{2}}$,  $\glfibre\simeq Z_{1}\times\cdots\times Z_{n-1}$ as algebraic varieties.
\end{cor}

\bibliographystyle{plain.bst}

\bibliography{bibliography}

\end{document}